\documentclass{article}

\usepackage{amssymb, amsfonts, amsmath, amscd, amsthm} 
\usepackage{hyperref} 
\usepackage{xcolor, comment, multicol}
\usepackage[misc]{ifsym}

\theoremstyle{plain}
\newtheorem{thm}{Theorem}[section]
\newtheorem*{thm*}{Theorem}
\newtheorem{cor}[thm]{Corollary}
\newtheorem{prop}[thm]{Proposition}

\theoremstyle{definition}

\newtheorem{defn}[thm]{Definition}
\newtheorem{rem}[thm]{Remark}

\title{Orlicz Space Interpolation and Its Applications to Operator Convolution}
 \author{Wolfram Bauer, Robert Fulsche, Joachim Toft}

\begin{document}

\maketitle
\begin{abstract}
    We prove a strong-type interpolation result for noncommutative Orlicz spaces over semifinite von Neumann algebras. Based on this result, we obtain Young-type convolution estimates for the Weyl pseudodifferential symbols of operators in appropriate Orlicz-Schatten spaces. Equivalently, we prove convolution estimates of Young type for Werner's function-operator convolutions in quantum harmonic analysis. 

    \emph{2010 Mathematics Subject Classification.}  Primary: 46E30, 35S05, 47G30; Secondary: 47B10, 46L52

    \emph{Keywords:} Pseudodifferential operators, Weyl quantization, quantum
harmonic analysis, noncommutative Orlicz spaces
\end{abstract}

\section{Introduction}

The present paper is concerned with convolution estimates of operators and functions (in the sense of Werner's \emph{quantum harmonic analysis} \cite{werner84}), or equivalently with convolution estimates of pseudodifferential symbols of operators (cf.~\cite{Toft2002}). The $L^p$-versions of Young's convolution estimates for functions and operators, respectively for pseudodifferential symbols, have now been known for many years. In the present paper, we will extend these convolution estimates to spaces of Orlicz-type.

While there are connections between some of the applications considered here and the results in \cite{Bauer_Fulsche_Toft2025}, the main focus of the present paper is the development of interpolation theory for Orlicz spaces. In contrast to \cite{Bauer_Fulsche_Toft2025}, which relies on direct computations of Wigner distributions, the approach adopted here is more abstract, based on interpolation results for noncommutative Orlicz spaces over semifinite von Neumann algebras. This framework allows for a range of applications beyond those discussed in \cite{Bauer_Fulsche_Toft2025}.

For precise definitions of the objects used here we refer to the respective parts of the paper. Let $\mathcal M$ be a semifinite von Neumann algebra. By $L^0(\mathcal M)$ we denote the $\ast$-algebra of measurable operators. By $L^p(\mathcal M)$ and $L^p_w(\mathcal M)$ we denote the noncommutative $L^p$, respectively weak $L^p$ spaces affiliated with the von Neumann algebra. If $\mathcal N$ is another semifinite von Neumann algebra, a quasi-linear operator $T: L^0(\mathcal M) \to L^0(\mathcal N)$ is said to be of \emph{weak type} $(p,q)$ (where $0 < p \leq q \leq \infty)$ if there exists $C > 0$ such that for all $x \in L^p(\mathcal M)$ it holds true that:
\begin{align*}
    \| Tx\|_{L^q_w} \leq C \| x\|_{L^p}.
\end{align*} 
Further, $T$ is said to be of \emph{strong type} $(p, q)$ if there exists $D > 0$ such that:
\begin{align*}
    \| Tx\|_{L^q} \leq D \| x\|_{L^p}
\end{align*}
for all $x \in L^p(\mathcal M)$. If now $\Phi$ is a \emph{Young function}, we denote by $L^\Phi(\mathcal M)$ the affiliated noncommutative Orlicz space. By $p_\Phi, q_\Phi$ we denote the characteristic exponents of $\Phi$ (see below for details). Now, our main result regarding interpolation of noncommutative Orlicz spaces reads as follows:

\begin{thm*}
 Assume that $\Phi$ is a quasi-Young function satisfying $0 < p_0 < q_{\Phi}  \leq p_{\Phi} < p_1 \leq \infty$ and $T: L^0(\mathcal{M})  \rightarrow L^0(\mathcal{N})$ be a quasilinear operator which is of weak  type  $(p_i,p_i)$ for $i=0,1$ if $p_1< \infty$ and of strong type $(p_1,p_1)$ if $p_1= \infty$. Then $T$ is of $s(\Phi, \Phi)$-type, i.e.~there is $d>0$ such that 
\begin{equation*}
 \| Tx \|_{L^{\Phi}} \leq d \| x \|_{L^{\Phi}}, \hspace{5ex} \forall x \in L_{\Phi}(\mathcal{M}). 
 \end{equation*}
 \end{thm*}
 We want to emphasize that a related result, giving a $w(\Phi, \Phi)$-type estimate, was obtained in \cite{Bekjan_Chen_Liu_Jiao2011}.

Our main motivation, and also main application, of investigating this interpolation result, are Young-type convolution estimates for symbols spaces of pseudodifferential operators. A detailed discussion of possible applications can also be found in the paper \cite{Bauer_Fulsche_Toft2025}. Let us briefly summarize the main findings of the present paper. 

In the following, we will denote by $L^\psi(\mathbb R^{2d})$ the Orlicz space on $\mathbb R^{2d}$ and by $s_\psi^w$ the class of Weyl symbols of all operators in the Orlicz-Schatten class $S^\psi(L^2(\mathbb R^d))$, where $\psi$ usually denotes a Young function. We refer to subsequent sections of this paper for precise definitions. The first principal result is the following:

\begin{thm*}
Let $\psi_0, \psi_1, \psi_2$ be Young functions with $1 < q_{\psi_j} \leq p_{\psi_j} < \infty$ for $j = 0, 1, 2$ and for all $s>0$ satisfying the relation
\begin{equation}
    s\psi_0^{-1}(s) = \psi_1^{-1}(s) \psi_2^{-1}(s). 
\end{equation}
Further, assume that $p_{\psi_1}^{-1} + p_{\psi_2}^{-1} > 1$ (this notation will be introduced later). Then convolutions act as continuous bilinear operators: 
\begin{align*}
    &\ast: L^{\psi_2}(\mathbb R^{2d}) \times L^{\psi_1}(\mathbb R^{2d}) \to L^{\psi_0}(\mathbb R^{2d}),\\
    &\ast: s_{\psi_2}^w \times L^{\psi_1}(\mathbb R^{2d}) \to s_{\psi_0}^w,\\
    &\ast: s_{\psi_2}^w \times s_{\psi_1}^w \to L^{\psi_0}(\mathbb R^{2d}).
\end{align*}
\end{thm*}
Our second main theorem is the following, where we denote for a function $a: \mathbb R^{2d} \to \mathbb C$ and $t \in \mathbb R$ by $a_t$ the dilated function $a_t(z) = a(tz)$.
\begin{thm*}
Let $\psi_0, \psi_1, \psi_2$ be Young functions with $\Delta_2$ and for all $s>0$ satisfying the relation
\begin{equation*}
    s\psi_0^{-1}(s) = \psi_1^{-1}(s) \psi_2^{-1}(s). 
\end{equation*}
Further, assume that $p_{\psi_1}^{-1} + p_{\psi_2}^{-1} > 1$. Let $0 \neq s, t \in \mathbb R$ such that $\pm s^{-2} \pm t^{-2} = 1$ for some combination of $+$ and $-$. Then, the dilated convolution
    \begin{align*}
        (a, b) \mapsto a_s \ast b_t
    \end{align*}
    extends to a continuous bilinear map
    \begin{align*}
        s_{\psi_2}^w \times s_{\psi_1}^w \to s_{\psi_0}^w.
    \end{align*}
\end{thm*}

The paper is organized as follows: In Section \ref{sec:interpolation}, we discuss noncommutative Orlicz spaces and a certain interpolation theorem on these spaces, which is one of the main technical tools in the proofs of our results. Section \ref{sec:3} is dedicated to the formulation and proof of the first theorem mentioned above. Indeed, we will not only prove the result mentioned earlier, but a multi-linear version of this theorem, describing mapping properties of iterated convolution. In Section \ref{sec:4} we will then deal with our second main theorem, i.e., the result concerning dilated convolution. Just as for the first result, we will also provide a more general version than the one mentioned above, yielding mapping properties of iterated dilated convolutions. We added Section \ref{sec:5} which describes how the results we obtained translate into the language of function-operator convolutions in R.~Werner's quantum harmonic analysis \cite{werner84}. Finally, in Section \ref{sec:discussion}, we added a short discussion on possible generalizations of our results.

\section{Interpolation in noncommutative Orlicz spaces}\label{sec:interpolation}
We say that $\Phi: [0, \infty) \rightarrow [0, \infty]$ is a Young function (in the sense of \cite{Bekjan_Chen_Liu_Jiao2011,Bonino_etal2023}) if: 
 \begin{itemize}
 \item[(a)] $\Phi$ is convex and increasing,
 \item[(b)] $\Phi(0)=0$, 
 \item[(c)] $\lim_{t \rightarrow \infty} \Phi(t)=+ \infty$. 
 \end{itemize}
Further, we say that a function $\Phi: [0, \infty) \rightarrow [0, \infty]$ is a \emph{quasi-Young function} (or \emph{$r$-Young function}) if there exists some $0 < r \leq 1$ and a Young function $\Phi_0$ such that $\Phi(t) = \Phi_0(t^r)$ for all $t \geq 0$.

 Recall that $\Phi$ satisfies the {\it $\Delta_2$-condition} if there is some $C>0$ such that for all $t>0$: 
 \begin{equation}\label{GL-Delta-2-condition}
 \Phi(2t) \leq C \Phi(t).
 \end{equation}
 Throughout the paper we frequently assume (\ref{GL-Delta-2-condition}) to hold. It is equivalent to assuming for any $\alpha>0$ the existence of a constant $C_{\alpha}>0$ such that 
 \begin{equation*}
\Phi(\alpha t) \leq C_{\alpha} \Phi(t) \hspace{4ex} \forall \; t>0.  
 \end{equation*}
 We write $\Phi_+^{\prime}(t)$ for the derivative of $\Phi$ from the right and define:
 \begin{equation}\label{GL_definition_q_Phi_p_Phi}
 q_{\Phi}:= \inf_{t>0} \frac{t \Phi_+^{\prime}(t)}{\Phi(t)} \hspace{4ex} \mbox{\it and } \hspace{4ex} p_{\Phi}:= \sup_{t>0} \frac{t \Phi_+^{\prime}(t)}{\Phi(t)}. 
 \end{equation}
 Then it is well-known that $0 < q_{\Phi} \leq p_{\Phi} \leq \infty$ (and $1 \leq q_\Phi$ whenever $\Phi$ is a Young function). We have an equivalent characterization of these numbers \cite{Bekjan_Chen_Liu_Jiao2011}: 
 \begin{align}\label{GL_equivalent_Characterization_exponent}
 q_{\Phi}&= \sup \Big{\{} p>0 \: : \: \frac{\Phi(t)}{t^p} \; \textup{is nondecreasing  } \forall \: t>0\Big{\}},\\
 p_{\Phi}&= \inf \Big{\{} q>0 \: : \: \frac{\Phi(t)}{t^q} \; \textup{is nonincreasing }\forall  \: t>0\Big{\}}. \label{GL_equivalent_Characterization_exponent2}
 \end{align}
 \begin{rem}
     We remark that $\Phi$ satisfies the $\Delta_2$-condition if and only if $p_{\Phi} < \infty$. 
 \end{rem}
Let $\mathcal{M}$ be a semifinite von Neumann algebra of operators acting on a Hilbert space $\mathcal H$ with positive part $\mathcal M_+$ and 
 equipped with a normal, semifinite faithful trace $\tau$. We recall the definition of these terms (see, e.g., \cite{Pisier_Xu2003}):
\begin{itemize}
    \item $\mathcal M$ being semifinite means that the identity can be written as a sum of mutually orthogonal finite projections.
    \item A trace $\tau$ is given by a map $\tau: \mathcal M_+ \to [0, \infty]$ satisfying $\tau(x^\ast x) = \tau(x x^\ast)$ for all $x \in \mathcal M$. Upon considering $\mathcal M_{\tau} = \{ x \in \mathcal M_+: ~\tau(x) < \infty\}$, the $\tau$ is extended by linearity to $\operatorname{span}\mathcal M_\tau$. Then, for all $x, y \in \operatorname{span}\mathcal M_\tau$, the trace satisfies $\tau(xy) = \tau(yx)$. 
    \item A trace is called \emph{normal} if for every $x \in \mathcal M_+$ and each net $(x_\gamma)_{\gamma \in \Gamma} \subset \mathcal M_+$ increasing towards $x$, it is $\tau(x) = \sup_{x \in \Gamma} \tau(x_\gamma)$.
    \item A trace is called \emph{faithful} if $\tau(x^\ast x) = 0$ implies $x = 0$.
    \item Finally, a trace is called \emph{semifinite} if for each $x \in \mathcal M_+$ there exists an increasing net $(x_\gamma)_{\gamma \in \Gamma} \subset \mathcal M_\tau$ converging strongly to $x$.
\end{itemize}
As is well-known, a von Neumann algebra is semifinite if and only if it admits a normal, faithful and semifinite trace on it.

 We write $L^0(\mathcal{M})=L^0(\mathcal{M}, \tau)$ for the $*$-algebra of measurable operators with respect to $(\mathcal{M}, \tau)$ 
 (see \cite{Fack_Kosaki1986} or \cite[Chapter 4]{Hiai2021} for definitions and further details). 
 \begin{defn}
 Let $x \in L^0(\mathcal{M})$ and by $e_{(\alpha, \infty)}(|x|)$ for $\alpha >0$ we denote the spectral projection of $|x|$ associated to the interval $(\alpha, \infty)$. With $t>0$ 
 we write: 
 \begin{align*}
 \lambda_{\alpha}(x):&= \tau \big{(} e_{(\alpha, \infty)}(|x|) \big{)}, \\
 \mu_t(x):&= \inf \big{\{} s >0 \: : \: \lambda_s(x)\leq t \big{\}}. 
 \end{align*}
 \end{defn}
 The maps $\alpha \mapsto \lambda_{\alpha}(x)$ and $t \mapsto \mu_t(x)$ are decreasing and continuous from the right. 

 With respect to these quantities the non-commutative weak $L^p$ spaces $L^p_w(\mathcal{M})$, where $0<p<\infty$, can be defined as the collection of  $x \in L^0(\mathcal{M})$ 
 for which 
 \begin{equation*}
 \|x \|_{L^p_w}:= \sup_{t>0} t^{\frac{1}{p}} \mu_t(x) < \infty.
 \end{equation*}
 We refer to \cite{Bekjan_Chen_Liu_Jiao2011} for some other useful descriptions of $\| \cdot \|_{L^p_w}$. In particular, by \cite[Eq.\ (2.4)]{Bekjan_Chen_Liu_Jiao2011}, we have
 \begin{align}\label{char:weaknorm}
     \| x\|_{L^p_w} = \sup_{s > 0} s\lambda_s(x)^{1/p}.
 \end{align}
 Moreover, the {\it non-commutative Orlicz space} is defined as 
 \begin{equation*}
 L^{\Phi}(\mathcal{M}):= \Big{\{} x \in L^0(\mathcal{M}) \: : \:  \exists\: c>0, \: N(c):=\int_0^{\infty} \left[t \Phi \left( \frac{\mu_t(x)}{c} \right)\right] \frac{dt}{t} < \infty \Big{\}}
 \end{equation*}
 equipped with the quasi-norm: 
 \begin{equation*}
 \|x\|_{L^{\Phi}} := \inf \Big{\{} c>0 \: :\:  N(c) \leq 1 \Big{\}}.  
  \end{equation*}
  In particular, in the case $\Phi(t)=t^p$ where $p > 0$, we have 
  \begin{equation*}
  \|x\|_{L^p}^p= \int_0^{\infty} \mu_t(x)^p dt. 
  \end{equation*}
 Similarly, the noncommutative weak Orlicz space can be defined as
 \begin{align*}
      L^\Phi_w(\mathcal M) := \{ x \in L^0(\mathcal M): ~\exists c > 0 \text{ s.th.\ } \sup_{t > 0} t \Phi(\mu_t(x)/c) < \infty\},
 \end{align*}
 which is endowed with
 \begin{align*}
     \| x\|_{L^\Phi_w} = \inf \{ c > 0:~ t\Phi(\mu_t(x)/c) \leq 1 \text{ for each } t > 0\}.
 \end{align*}
 We note that $\|x\|_{L^{\Phi}_w} \leq \|x\|_{L^{\Phi}}$ and 
 $L^{\Phi}(\mathcal{M}) \subset L^{\Phi}_w(\mathcal{M})$ (see \cite[Proposition 3.2(4)]{Bekjan_Chen_Liu_Jiao2011}).
 
 These notions generalize the classical Orlicz spaces $L^{\Phi}$ and weak Orliczs spaces $L_w^{\Phi}$ on a complete non-atomic measure space as treated in \cite{Bonino_etal2023, Liu_Wang2013}.
 Non-commutative (weak) 
 Orlicz spaces also generalize the non-commutative (weak) $L^p$-spaces in \cite{Pisier_Xu2003}, where $1 \leq p \leq \infty$. We write $L^p(\mathcal{M})$  
 instead of $L^{\Phi}(\mathcal{M})$ 
 in the case $\Phi(t)=t^p$ with $0 < p< \infty$. Moreover,  we put 
 $L_w^{\infty}(\mathcal{M})= L^{\infty}(\mathcal{M}):= \mathcal{M}$ equipped with the operator norm. 

Let  $\mathcal{N}$ be a second semifinite von Neumann algebra with normal semifinite faithful trace $\nu$.
 Below we prove an interpolation theorem for quasi-linear operators. We consider operators
 \begin{equation}\label{GL_quasilinear_operator}
 T: L^0(\mathcal{M})  \rightarrow L^0(\mathcal{N}). 
 \end{equation}
Recall the following definition (see \cite{Bekjan_Chen_Liu_Jiao2011}): 
 \begin{defn}
 The operator $T$ in \eqref{GL_quasilinear_operator} is called {\it quasilinear} if 
 \begin{itemize}
 \item[(a)] $|T(\lambda x) | \leq |\lambda| |Tx|$ for all $x \in L^0(\mathcal{M})$ and all $\lambda \in \mathbb{C}$.
 \item[(b)] There is $K>0$ such that for all $x,y \in L^0(\mathcal{M})$ there exist partial isometries $u,v \in \mathcal{N}$ such that 
 \begin{equation}\label{GL-inequality-quasi-linear-operator} 
 |T(x+y)| \leq K \big{(} u^* |Tx| u+v^*|Ty| v \big{)}. 
 \end{equation}
 \end{itemize}
 \end{defn}  
 A quasilinear operator $T: L^0(\mathcal{M})  \rightarrow L^0(\mathcal{N})$ is said to be of {\it weak type} $(p,q)$ with $0 <p \leq q\leq \infty$ if there is $C>0$ such that for all $x \in L_p(\mathcal{M})$:
 \begin{equation}\label{GL_weak-type-estimate}
 \|Tx \|_{L^{q}_w} \leq C\|x\|_{L^{p}}.
 \end{equation}
 We call $T$ of {\it strong type} $(p,q)$ if there is $D>0$ such that 
 \begin{equation*}
     \|Tx\|_{L^q} \leq D \|x \|_{L^{p}}.
 \end{equation*}
\cite[Theorem 4.2]{Bekjan_Chen_Liu_Jiao2011} yields the following result:
\begin{thm}\label{thm:weak_type_interpolation}
    Let $T: L^0(\mathcal M, \tau) \to L^0(\mathcal N, \nu)$ be quasilinear and $\Phi$ a Young function. Further, let $0 < p_0 < q_\Phi \leq p_\Phi < p_1 \leq \infty$. Assume that $T$ is of weak types $(p_0,p_0)$ and $(p_1, p_1)$ if $p_1 < \infty$ or of weak type $(p_0, p_0)$ and of strong type $(p_1, p_1)$ if $p_1 = \infty$. Then, there exists $C > 0$ such that
    \begin{align*}
        \| Tx\|_{L^\Phi_w}\leq C \| x\|_{L^\Phi_w}
    \end{align*}
    for every $x \in L^\Phi_w(\mathcal M)$.
\end{thm}
On the other hand, \cite{Liu_Wang2013} obtains an analogous estimate of strong type for Orlicz spaces of commutative von Neumann algebras. We discuss now the general strong type interpolation result in the non-commutative setting. We point out that the proof follows by a suitable modification of the proof for the commutative statement.

\begin{thm}\label{thm:interpolation_strong_type}
 Assume that $\Phi$ is a quasi-Young function with $0 < p_0 < q_{\Phi}  \leq p_{\Phi} < p_1 \leq \infty$ and $T: L^0(\mathcal{M})  \rightarrow L^0(\mathcal{N})$ be a quasilinear operator which is of weak  type  $(p_i,p_i)$ for $i=0,1$ if $p_1< \infty$ and of strong type $(p_1,p_1)$ if $p_1= \infty$. Then $T$ is of $s(\Phi, \Phi)$-type, i.e.~there is $d>0$ such that 
\begin{equation*}
 \| Tx \|_{L^{\Phi}} \leq d \| x \|_{L^{\Phi}}, \hspace{5ex} \forall x \in L^{\Phi}(\mathcal{M}). 
 \end{equation*}
 \end{thm}
 \begin{rem}
     Indeed, for the application we have in mind, it is not suitable to start with operators acting as $T: L^0(\mathcal M) \to L^0(\mathcal N)$. Instead, we will have to consider operators mapping
     \begin{align*}
         T: L^{p_0}(\mathcal M) + L^{p_1}(\mathcal M) \to L^{p_0}(\mathcal N) + L^{p_1}(\mathcal N).
     \end{align*}
     Nevertheless, the statement of the theorem, as well as the proof, carry over verbatim to this setting.
 \end{rem}
 \begin{proof} 
 Let $x \in L^{\Phi}(\mathcal{M})$ with $\|x\|_{L^{\Phi}}=1$. For any $\alpha >0$ we decompose $x= x_{\alpha}+ x^{\alpha}$, where: 
 \begin{equation*}
 x_{\alpha}:=x e_{(\alpha, \infty)}(|x|) \hspace{4ex} \mbox{\it and } \hspace{4ex} x^{\alpha} := x e_{[0,\alpha]}(|x|). 
 \end{equation*}
Let $K>0$ denote the constant in \eqref{GL-inequality-quasi-linear-operator}.  It has been shown in the proof of \cite[Theorem 4.2]{Bekjan_Chen_Liu_Jiao2011} that 
 \begin{equation} \label{GL_inequality-distribution-function}
 \lambda_{2K\alpha}(Tx) \leq \lambda_{\alpha} \big{(}Tx_{\alpha}\big{)} + \lambda_{\alpha} \big{(} Tx^{\alpha} \big{)}.
 \end{equation}
We estimate the following integral: 
 \begin{align*}
& \int_0^{\infty} \Phi \big{(} \mu_t(Tx) \big{)} dt
 = \int_0^{\infty} \textup{vol} \Big{(} t>0 \: : \: \mu_t \big{(} Tx) > \alpha \Big{)} d\Phi(\alpha).
 \end{align*}
 Here $\textup{``vol''}$ means the volume with respect to the Lebesgue measure on  $\mathbb{R}_+=[0, \infty)$. Note that $\mu_t(Tx)\geq \alpha$ if and only if 
 $\lambda_{\alpha}(Tx) \geq t$ and therefore we have
 \begin{equation*}
 \textup{vol} \Big{(} t>0 \: : \: \mu_t \big{(} Tx) > \alpha \Big{)}= \lambda_{\alpha}(Tx)
 \end{equation*}
 proving that 
 \begin{align}
\int_0^{\infty} \Phi \big{(} \mu_t(Tx) \big{)} dt
&=  \int_0^{\infty}\lambda_{\alpha}(Tx) d \Phi(\alpha) \notag\\
& \leq   \int_0^{\infty}\lambda_{\alpha/2K}\big{(} Tx_{\alpha/2K} \big{)} d\Phi(\alpha) + \int_0^{\infty}\lambda_{\alpha/2K}\big{(} Tx^{\alpha/2K} \big{)} d\Phi(\alpha). 
\label{GL-quasilinear-interpolation}
 \end{align}
 Here we have used \eqref{GL_inequality-distribution-function} in the second line. In what follows we have by \eqref{char:weaknorm} for all $\alpha>0$: 
 \begin{equation*}
 \lambda_{\alpha/2K}(Tx_{\alpha/2K})\leq \frac{(2K)^{p_0}}{\alpha^{p_0}} \|Tx_{\alpha/2K}\|_{L^{p_0}_w}^{p_0}.
 \end{equation*}
 We estimate both summands on the right hand side of \eqref{GL-quasilinear-interpolation} separately. Using \eqref{GL_weak-type-estimate} with $q=p=p_0$ and $C=C_0>0$ we obtain: 
 \begin{align*}
  \int_0^{\infty}\lambda_{\alpha/2K}\big{(} Tx_{\alpha/2K}\big{)} d\Phi(\alpha) 
  & \leq (2K)^{p_0} \int_0^{\infty} \frac{\| Tx_{\alpha/ 2K}\|^{p_0}_{L^{p_0}_w}}{\alpha^{p_0}}d\Phi(\alpha) \\
  &\leq  (2KC_0)^{p_0}  \int_0^{\infty} \frac{\| x_{\alpha/ 2K}\|^{p_0}_{L^{p_0}}}{\alpha^{p_0}}d\Phi(\alpha) \\
  &=(2KC_0)^{p_0} \int_0^{\infty}  \frac{1}{\alpha^{p_0}} \int_0^{\infty} \mu_t(x_{\alpha/ 2K})^{p_0} dt ~d\Phi(\alpha)=(*).
   \end{align*}
We calculate the integrand: note that for all $\alpha >0$: 
\begin{equation*}
\lambda_s(x_{\alpha})= 
\begin{cases}
\lambda_{\alpha}(x) & \text{\it if } \; s \leq \alpha\\
\lambda_s(x) & \text{\it if } \: s \geq \alpha. 
\end{cases}
\end{equation*} 
Hence it follows \footnote{ Clearly $\lambda_s(x_{\alpha}) \leq \lambda_s(x)$ for all $s>0$ showing that $\mu_t(x_{\alpha}) \leq \mu_t(x)$ for all $t>0$. If $\lambda_{\alpha}(x)>t$, then 
$\lambda_s(x_{\alpha})> t$ for $s \leq \alpha$. Therefore $\mu_t(x_{\alpha})= \inf\{ s> \alpha \:: \:  \lambda_s(x_{\alpha}) \leq t\}=\inf \{ s>\alpha \: : \: \lambda_s(x) \leq t\} 
\geq \mu_t(x)$.}
\begin{equation*}
\mu_t(x_{\alpha})= \inf\{ s \geq 0 \: : \: \lambda_s(x_{\alpha}) \leq t \}=
\begin{cases}
0 & \textup{\it if } \; \lambda_{\alpha}(x) \leq t,\\
\mu_t(x) & \textup{\it if } \: \lambda_{\alpha}(x) >t.
\end{cases}
\end{equation*}
In conclusion: 
\begin{align*}
(*)=(2KC_0)^{p_0} \int_0^{\infty}\frac{1}{\alpha^{p_0}} \int_0^{\lambda_{\alpha/2K}(x)} \mu_t(x)^{p_0} dt ~d \Phi(\alpha). 
\end{align*}
By $\chi_A$ we denote the characteristic function of a set $A \subset \mathbb{R}_+$. Then, by definition we have 
$\chi_{[0, \lambda_{\alpha}(x)]}(t)= \chi_{[0, \mu_t(x)]}(\alpha)$ and therefore we can write by Tonellis theorem: 
\begin{align*}
(*)&=(2KC_0)^{p_0} \int_0^{\infty} \mu_t(x)^{p_0} \int_0^{\mu_t(x)} \frac{\Phi_+^{\prime}(\alpha)}{\alpha^{p_0}} d\alpha ~dt\\
&= (2KC_0)^{p_0} \int_0^{\infty} \mu_t(x)^{p_0} \int_0^{\mu_t(x)} \frac{\Phi(\alpha)}{\alpha^{q_{\Phi}}} \cdot \frac{\alpha \Phi_+^{\prime}(\alpha)}{\Phi(\alpha)} \alpha^{q_{\Phi}-p_0-1} d\alpha ~dt \\
& \leq p_{\Phi}(2KC_0)^{p_0} \int_0^{\infty} \mu_t(x)^{p_0} \frac{\Phi(\mu_t(x))}{\mu_t(x)^{q_{\Phi}}} \int_0^{\mu_t(x)} \alpha^{q_{\Phi}-p_0-1} d\alpha ~dt\\
&= \frac{p_{\Phi}}{q_{\Phi}-p_0} (2KC_0)^{p_0} \int_0^{\infty} \Phi\big{(} \mu_t(x) \big{)} dt 
\leq  \frac{p_{\Phi}}{q_{\Phi}-p_0} (2KC_0)^{p_0}, 
\end{align*}
where we have used the definition \eqref{GL_definition_q_Phi_p_Phi} of $p_{\Phi}$ and  the fact that the assignment 
\begin{equation*}
(0, \infty) \ni \alpha \mapsto \frac{\Phi(\alpha)}{\alpha^{q_{\Phi}}}
\end{equation*}
is non-decreasing as well as the assumption $\|x\|_{L^{\Phi}}=1$ (in the last inequality). 
\vspace{1ex}\par 
First, we assume that $p_1< \infty$, and we estimate the second term on the right of \eqref{GL-quasilinear-interpolation} in a similar way. Using \eqref{GL_weak-type-estimate} with $q=p=p_1$ and $C=C_1>0$ we obtain: 
\begin{align*}
 \int_0^{\infty}\lambda_{\alpha/2K}\big{(} Tx^{\alpha/2K} \big{)} d\Phi(\alpha) 
 &\leq  (2K)^{p_1} \int_0^{\infty} \frac{\|Tx^{\alpha/2K}\|_{L^{p_1}_w}^{p_1}}{\alpha^{p_1}} d\Phi(\alpha)\\
 &\leq (2KC_1)^{p_1} \int_0^{\infty}\frac{\|x^{\alpha/2K}\|_{L^{p_1}}^{p_1}}{\alpha^{p_1}} d\Phi(\alpha)\\
  &=(2KC_1)^{p_1} \int_0^{\infty}  \frac{1}{\alpha^{p_1}} \int_0^{\infty} \mu_t(x^{\alpha/ 2K})^{p_1} dt ~d\Phi(\alpha)=(**).
\end{align*}
Note that $\lambda_s(x^{\alpha})= 0$ if $s >\alpha$ and hence $\mu_t(x^{\alpha}) \leq \alpha$ for all $t>0$. We estimate the inner integral from above: 
\begin{align*}
\int_0^{\infty} \mu_t(x^{\alpha/ 2K})^{p_1} dt 
&= p_1\int_0^{\infty} \lambda_s(x^{\alpha/2K}) s^{p_1-1} ds\\
&=p_1\int_0^{\alpha/2K} \lambda_s(x^{\alpha/2K}) s^{p_1-1} ds\leq p_1\int_0^{\alpha/2K} \lambda_s(x) s^{p_1-1} ds. 
\end{align*}
Inserting above shows: 
\begin{align*}
(**)& \leq p_1 (2KC_1)^{p_1} \int_0^{\infty} \int_0^{\alpha/2K}\alpha^{-p_1} \lambda_s(x) s^{p_1-1} ds ~d\Phi(\alpha)\\ 
&= p_1 (2KC_1)^{p_1} \int_0^{\infty}\lambda_s(x) s^{p_1-1} \int_{2Ks}^{\infty} \alpha^{-p_1}\Phi_+^{\prime}(\alpha) d\alpha ~ds\\
&=p_1 (2KC_1)^{p_1} \int_0^{\infty}\lambda_s(x) s^{p_1-1} \int_{2Ks}^{\infty} \frac{\alpha\Phi_+^{\prime}(\alpha)}{\Phi(\alpha)} \cdot
 \frac{\Phi(\alpha)}{\alpha^{p_{\Phi}}}   \alpha^{p_{\Phi}-p_1-1} d\alpha ~ds. 
\end{align*}
Using the definition of $p_{\Phi}$ in (\ref{GL_definition_q_Phi_p_Phi}) and the fact that 
\begin{equation*}
(0, \infty) \ni \alpha \mapsto  \frac{\Phi(\alpha)}{\alpha^{p_{\Phi}}} 
\end{equation*}
is non-increasing shows: 
\begin{align*}
(**)& \leq p_1p_{\Phi} (2KC_1)^{p_1} \int_0^{\infty}\lambda_s(x) s^{p_1-1}\frac{\Phi(2Ks)}{(2Ks)^{p_{\Phi}}} \int_{2Ks}^{\infty} \alpha^{p_{\Phi}-p_1-1} d\alpha ~ds\\
&= \frac{p_1p_{\Phi}}{p_1-p_{\Phi}}2KC_1^{p_1} \int_0^{\infty}\lambda_s(x) \underbrace{\frac{\Phi(2Ks)}{(2Ks) \Phi_+^{\prime}(2Ks)} }_{\leq \frac{1}{q_{\Phi}}}\Phi^{\prime}_+(2Ks)ds\\
&=  \frac{p_1p_{\Phi}}{q_{\Phi}(p_1-p_{\Phi})}C_1^{p_1}\int_0^{\infty} \Phi \big{(} 2K \mu_t(x)) dt\\
&\leq \frac{p_1p_{\Phi}}{q_{\Phi}(p_1-p_{\Phi})}C_1^{p_1}C_K\int_0^{\infty} \Phi \big{(} \mu_t(x)) dt\leq  \frac{p_1p_{\Phi}}{q_{\Phi}(p_1-p_{\Phi})}C_1^{p_1}C_K. 
\end{align*}
Here we choose $C_K>0$ with $\Phi(2Kt) \leq C_K \Phi(t)$ for all $t>0$. Moreover, in the last inequality we have used $\| x\|_{L^{\Phi}}=1$. Combing the above inequalities shows: 
\begin{equation}\label{theorem-interpolation-Orlicz-GL}
\int_0^{\infty} \Phi \big{(} \mu_t(Tx) \big{)} dt \leq  \frac{p_{\Phi}}{q_{\Phi}-p_0} (2KC_0)^{p_0}+ \frac{p_1p_{\Phi}}{q_{\Phi}(p_1-p_{\Phi})}C_1^{p_1}C_K. 
\end{equation}
\par 
By $C>0$ we denote the maximum of 1 and the right hand side. Since $\Phi(0)=0$, $\Phi(t) = \Phi_0(t^r)$ with $r \in (0, 1]$ and $\Phi_0$ is convex we conclude from $1/C\leq 1$: 
\begin{equation}\label{GL_inequality_convexity}
\Phi\Big{(}\frac{\mu_t(Tx)}{C}\Big{)} \leq \frac{1}{C^r} \Phi(\mu_t(Tx))
\end{equation} 
and therefore, according to (\ref{theorem-interpolation-Orlicz-GL}): 
\begin{equation*}
N(C)=\int_0^{\infty} \Phi \Big{(} \frac{\mu_t(Tx)}{C} \Big{)} dt\leq \frac{1}{C^r} \int_0^{\infty} \Phi \big{(} \mu_t(Tx) \big{)} dt \leq 1. 
\end{equation*}
Therefore $\|Tx\|_{L^{\Phi}} \leq C = C \|x\|_{L^{\Phi}}$ proving the statement for $p_1< \infty$. 
\vspace{1ex}\par
Assume now that $p_1= \infty$. Since $T$ is of strong type $(p_1,p_1)$ by assumption we have $C_1>0$ such that
\begin{equation*}
\|Tx^{\alpha}\|_{L^{\infty}}= \|Tx^{\alpha} \|_{\textup{op}} \leq C_1 \|x^{\alpha}\|_{\textup{op}} \leq  C_1 \big{\|} |x| e_{[0, \alpha]}(|x|)\big{\|} \leq C_1 \alpha. 
\end{equation*}
Replacing $T$ by $C_1^{-1} T$ we may, without loss of generality, assume that $C_1=1$. As a consequence we have
\begin{equation*}
\lambda_{\alpha/2K} \big{(} Tx^{\alpha/2K} \big{)} =\tau \Big{(} e_{(\alpha/2K, \infty)}(|Tx^{\alpha/2K}|) \Big{)}=0
\end{equation*}
and (\ref{GL-quasilinear-interpolation}) above shows: 
\begin{align*}
\int_0^{\infty} \Phi\big{(} \mu_t(Tx) \big{)} dt \leq \int_0^{\infty} \lambda_{\alpha/2K}\big{(} Tx_{\alpha/2K} \big{)} d\Phi(\alpha)\leq \frac{p_{\Phi}}{q_{\Phi}-p_0} (2KC_0)^{p_0}.
\end{align*}
Let $\widetilde{C}$ denote the maximum of $1$ and the quantity of the right hand side. Applying \eqref{GL_inequality_convexity} again shows: 
\begin{equation*}
\int_0^{\infty} \Phi\Big{(} \frac{\mu_t(Tx)}{\widetilde{C}} \Big{)} dt \leq 1. 
\end{equation*} 
It follows that $\|Tx\|_{L^{\Phi}} \leq \widetilde{C} = \widetilde{C} \|x\|_{L^{\Phi}}$, finishing the proof. 
 \end{proof} 

Note that the proof we gave yields explicit bounds on the norm of the interpolated operator. For a special instance of the result, we formulate these bounds explicitly:
\begin{cor}\label{cor:interpolation}
Assume that $\Phi$ is a Young function satisfying $1  < q_\Phi \leq p_\Phi < \infty$ and that
\begin{align*}
    T: L^1(\mathcal M) + L^\infty(\mathcal M) \to L^1(\mathcal N) + L^\infty(\mathcal N)
\end{align*}
is linear. Further, assume that $T$ is of weak type $(1, 1)$ and of strong type $(\infty, \infty)$. Then,
\begin{align*}
    T: L^\Phi(\mathcal M) \to L^\Phi(\mathcal N)
\end{align*}
is bounded with
\begin{align*}
    \| T\|_{L^\Phi \to L^\Phi} \leq C_1 \max \{ 1, 2\frac{p_\Phi}{q_\Phi - 1} C_0\},
\end{align*}
where
\begin{align*}
    C_0 = \| T\|_{L^1 \to L^1_w}, \quad C_1 = \| T\|_{L^\infty \to L^\infty}.
\end{align*}
\end{cor}

\section{Convolution estimates}\label{sec:3}

We still let $\Phi$ be a Young function satisfying $1  < q_\Phi \leq p_\Phi < \infty$ as in the previous setting.

We apply the interpolation results from above to two particular von Neumann algebras, namely $L^\infty(\mathbb R^{2d})$ and $\mathcal L(\mathcal H)$ for a separable, infinite-dimensional Hilbert space $\mathcal H$. 

Before going to the applications we have in mind, let us first put the conventions used earlier into perspective (see also \cite{Pisier_Xu2003}). 

For $\mathcal M = L^\infty(\mathbb R^{2d})$, the space $L^0(\mathcal M)$ simply agrees with the space of (equivalence classes of) all measurable functions on $\mathbb R^{2d}$ which are bounded outside sets of finite measure. The spectral projection $e_{(\alpha, \infty)}(|f|)$ is the characteristic function $\chi_{\{ |f| > \alpha\}}$. The trace on $L^\infty(\mathbb R^{2d})$ is defined through the Lebesgue integral, such that $\lambda_\alpha(f) = \operatorname{vol}(\{ |f| > \alpha\})$, where ``vol'' means the Lebesgue-measure of the set. Therefore, $\mu_t(f) = \inf \{ \alpha > 0: ~\operatorname{vol}( \{ |f| > \alpha\}) < t\}$ is the non-increasing rearrangement of $f$. Hence, the space $L^\Phi(L^\infty(\mathbb R^{2d})) =: L^\Phi(\mathbb R^{2d})$ is the usual commutative Orlicz space and $L^\Phi_w(L^\infty(\mathbb R^{2d})) =: L^\Phi_w(\mathbb R^{2d})$ the usual commutative weak Orlicz space.

Now, we consider $\mathcal M = \mathcal L(\mathcal H)$. Writing the self-adjoint operator $|A|$ in its spectral decomposition $|A| = \int_{[0, \infty)} \lambda ~d\pi(\lambda)$ with respect to the projection-valued measure $\pi$, we obtain (using the operator trace on $\mathcal L(\mathcal H)$):
\begin{align*}
    e_{(\alpha, \infty)}(|A|) &= \int_{(\alpha, \infty)} 1~d\pi(\lambda) = \pi((\alpha, \infty)),\\
    \lambda_\alpha(A) &= \operatorname{tr}(\pi((\alpha, \infty)).
\end{align*}
Hence, $\lambda_\alpha(A) < \infty$ if and only if $\pi((\alpha, \infty))$ is a projection with finite-dimensional range. In particular, for $\lambda_\alpha(A) < \infty$ it is necessary that $|A|$ is bounded. It turns out that $\mu_t(A) = \mu_{\lfloor t\rfloor}(A)$ are simply the singular values of $A$, and the spaces $L^p(\mathcal M)$ are the $p$-Schatten ideals (cf.\ \cite[Example 4.34]{Hiai2021}).

In the following, we will consider the Orlicz spaces $L^\Phi(\mathbb R^{2d})$ and the Orlicz-Schatten ideals $S^\Phi(\mathcal H) = L^\Phi(\mathcal L(\mathcal H))$, where we let $\mathcal H = L^2(\mathbb R^d)$. As is well-known, each $A \in \mathcal L(\mathcal H)$ has a Weyl symbol $\mathrm{sym}^w(A) \in \mathcal S'(\mathbb R^{2d})$, or equivalently $A = \mathrm{op}^w(f)$ for some $f \in \mathcal S'(\mathbb R^{2d})$. We will denote by $s_p^w$ the set of all Weyl symbols of operators in $S^p(\mathcal H)$, i.e., $s_p^w = \operatorname{sym}^w(S^p(\mathcal H))$, normed by 
$$ \| f\|_{s_p^w} := \| \mathrm{op}^w(f)\|_{S^p}, $$
and more generally, by $s_\Phi^w$ the set of Weyl symbols of all operators in $S^\Phi(\mathcal H)$, normed analogously. 

As is known \cite{Toft2002, werner84}, convolution continuously maps $(p \in [1, \infty])$:
\begin{align*}
    &\ast: s_1^w \times L^p(\mathbb R^{2d}) \to s_p^w,\\
    &\ast: s_p^w \times L^1(\mathbb R^{2d}) \to s_p^w,\\
    &\ast: s_p^w \times s_1^w \to L^p(\mathbb R^{2d}).
\end{align*}
The above continuity facts, together with Corollary \ref{cor:interpolation}, yield the following:
\begin{prop}\label{prop_1}
Let $\Phi$ be a Young function satisfying $1 < q_\Phi \leq p_\Phi < \infty$. Then, convolution extends to continuous maps
    \begin{align*}
        &\ast: L^1(\mathbb R^{2d}) \times L^\Phi(\mathbb R^{2d}) \to L^\Phi(\mathbb R^{2d}),\\
    &\ast: s_1^w \times L^\Phi(\mathbb R^{2d}) \to s_\Phi^w,\\
    &\ast: s_\Phi^w \times L^1(\mathbb R^{2d}) \to s_\Phi^w,\\
    &\ast: s_\Phi^w \times s_1^w \to L^\Phi(\mathbb R^{2d}),
    \end{align*}
    with norm estimates
    \begin{align*}
        \| f \ast g\|_{L^\Phi} &\leq 2\frac{p_\Phi}{q_\Phi - 1} \| g\|_{L^\Phi} \| f\|_{L^1},\\
        \| f \ast g\|_{s_\Phi^w} &\leq 2\frac{p_\Phi}{q_\Phi - 1} \| g\|_{L^\Phi} \| f\|_{s_1^w},\\
        \| f \ast g\|_{s_\Phi^w} &\leq 2\frac{p_\Phi}{q_\Phi - 1} \| f\|_{s_\Phi^w} \| g\|_{L^1},\\
        \| f \ast g\|_{L^\Phi} &\leq 2\frac{p_\Phi}{q_\Phi - 1} \| f\|_{s_\Phi^w} \| g\|_{s_1^w}.
    \end{align*}
\end{prop}
\begin{proof}
We only prove the statement for the first convolution, the other proofs work analogously. 

    Let $f \in L^1(\mathbb R^{2d})$ such that $\| f\|_{L^1} = 1$. Then, the map
    \begin{align*}
        T: L^1(\mathbb R^{2d}) + L^\infty(\mathbb R^{2d}) \to L^1(\mathbb R^{2d}) + L^1(\mathbb R^{2d})
    \end{align*}
    given by $Tg = f \ast g$ is of strong types $(1, 1)$ and $(\infty, \infty)$, both with norm constants $\| f\|_{L^1} = 1$. Hence, by Corollary \ref{cor:interpolation} we have
    \begin{align*}
        \| Tg\|_{L^\Phi} = \| f \ast g\|_{L^\Phi} \leq 2\frac{p_\Phi}{q_\Phi - 1} \| g\|_{L^\Phi}.
    \end{align*}
    For arbitrary $f \in L^1(\mathbb R^{2d})$, we therefore obtain:
    \begin{equation*}
        \| f \ast g\|_{L^\Phi} = \| f\|_{L^1} \| \frac{f}{\| f\|_{L_1}} \ast g\|_{L^\Phi}\leq 2\frac{p_\Phi}{q_\Phi - 1} \| g\|_{L^\Phi} \| f\|_{L^1}. \qedhere
    \end{equation*}
\end{proof}
In order to prove our next results we need some notation from 
\cite{Harjulehto_Hasto2019}. Recall that an increasing function $\varphi: [0,\infty) \rightarrow [0,\infty]=[0, \infty) \cup \{ \infty\}$ with $\varphi(0)=0$ is called {\it weak $\Phi$-function} if 
\begin{itemize}
\item[\textup{(i)}] $\lim_{t \downarrow 0} \varphi(t)=0$ and $\lim_{t \uparrow \infty} \varphi(t)= \infty$, 
\item[\textup{(ii)}] $\textup{(aInc)}_1$ holds on $(0, \infty)$, i.e. there is $a \geq 1$ such that for all $0< s <t$
$$\frac{\varphi(s)}{s} \leq a \frac{\varphi(t)}{t}.$$ 
\end{itemize}
Adapting the notation from \cite{Harjulehto_Hasto2019} we denote by $\Phi_w$ the set of {\it weak $\Phi$-functions}. Note that for a convex function $\varphi$ as above and 
with $0< s <t$ we obtain: 
 \begin{equation}\label{GL-convexity-and-one-decreasing}
\varphi(s)= \varphi \left( \frac{s}{t} t+0 \right) \leq \frac{s}{t} \varphi(t)+ \left( 1- \frac{s}{t} \right) \varphi(0)= \frac{s}{t} \varphi(t). 
\end{equation}
Hence, convexity implies (ii) with $a=1$. The set of left-continuous increasing convex functions $\varphi$ with $\varphi(0)=0$ and (i) therefore is contained in $\Phi_w$ and will 
be denoted by $\Phi_c \subset \Phi_w$. We extend a weak $\Phi$-functions to the interval $[0, \infty]$ by $\varphi(\infty)= \infty$ and we 
call $\varphi^{-1} : [0, \infty] \rightarrow [0, \infty]$ defined by
\begin{equation*}
\varphi^{-1}(\tau):= \inf \big{\{} t \geq 0 \: : \: \varphi(t) \geq \tau\big{\}}
\end{equation*} 
its {\it left-inverse} (although it may not be the left-inverse in an algebraic sense). If $\varphi \in \Phi_c$ satisfies $\Delta_2$, then it does not attain the value $\infty$ and therefore $\varphi$ is continuous. 

We recall the following fact on complex interpolation of Orlicz spaces, see \cite[Theorem 5.5.1]{Harjulehto_Hasto2019}. We just mention that the definitions of $L^{\varphi}(\mathbb R^{2d})$ as well as $s_{\varphi}^w$ for $\varphi \in \Phi_w$ are exactly as for $\varphi$ being a Young function.
\begin{thm}\label{thm:complex_interpolation_orlicz}
   Let $\varphi_0, \varphi_1$ be weak $\Phi$-functions.  
    Further, let $0 \leq \theta \leq 1$. Then, the function $[\varphi_0, \varphi_1]_{[\theta]}$ defined by
    \begin{align*}
        [\varphi_0, \varphi_1]_{[\theta]}^{-1} = (\varphi_0^{-1})^{1-\theta} (\varphi_1^{-1})^\theta,
    \end{align*}
  is also a weak $\Phi$-function and the following statements hold true for the complex interpolation method:
    \begin{itemize}
    \item[\textup{(a)}]  $[L^{\varphi_0}(\mathbb R^{2d}), L^{\varphi_1}(\mathbb R^{2d})]_{[\theta]} = L^{[\varphi_0, \varphi_1]_{[\theta]}}(\mathbb R^{2d})$,
    \item[\textup{(b)}]  $[s_{\varphi_0}^w, s_{\varphi_1}^w]_{[\theta]} = s_{[\varphi_0, \varphi_1]_{[\theta]}}^w$.
    \end{itemize}
\end{thm}
\begin{proof}
    The analogous interpolation result is certainly true more generally for complex interpolation between non-commutative Orlicz spaces $L^{\Phi_j}(\mathcal M)$, but we could not locate a suitable reference. We do not want to discuss the full proof here, but only give some references suitable for the case we want to apply later.

    For Orlicz spaces over commutative von Neumann algebras, such as $L^{\Phi}(\mathbb R^{2d})$ or $\ell^\phi(\mathbb Z)$ (the latter being $L^\Phi(\mathcal M)$ with $\mathcal M = \ell^\infty(\mathbb Z)$), the result can be found in \cite[Theorem 5.5.1]{Harjulehto_Hasto2019}. For Orlicz-Schatten ideals, i.e., $L^\Phi(\mathcal M)$ with $\mathcal M = \mathcal L(\mathcal H)$, $\mathcal H$ a Hilbert space, the statement follows from the one for $\ell^\Phi(\mathbb Z)$ by \cite[Theorem A.2.3]{Simon1976}. Since the spaces $s_\Phi^w$ are, by definition, isometrically isomorphic to the Orlicz-Schatten ideals $L^\Phi(\mathcal L(\mathcal H))$, the statement follows.
\end{proof}
\begin{rem}
    Clearly, for two weak $\Phi$-functions $\varphi_0, \varphi_1$ and $\theta \in [0, 1]$ we have the following identity, which we will make use of:
    \begin{align*}
        [\varphi_0, \varphi_1]_{[\theta]} = [\varphi_1, \varphi_0]_{[1-\theta]}.
    \end{align*}
\end{rem}
Note that the space $L^1(\mathbb R^{2d})$, respectively $s_1^w$, corresponds to the Young function $\Phi_{[1]}(t) = t$. Hence, for $\theta \in (0, 1)$ 
and $\Phi \in \Phi_w$ we have the following special case of the interpolated weak $\Phi$-function:
\begin{align}
    [\Phi, \Phi_{[1]}]_{[\theta]}^{-1}(s) = \Phi^{-1}(s)^{1-\theta} s^\theta.
\end{align}
When $\Phi_1, \Phi_2$ are Young functions satisfying $1 < q_{\Phi_j} \leq p_{\Phi_j} < \infty$, $j = 1, 2$, then so is $[\Phi_1, \Phi_2]_{[\theta]}$. 

We will make use of a more general approach to interpolation of Orlicz spaces, which consists of interpolating between $n$-tuples of Orlicz spaces. We will not refer to a general interpolation functor, but only make ad-hoc definitions in the setting of Orlicz spaces. The following approach iterates the usual complex interpolation method used above. Nevertheless, we are confident that this actually agrees with the complex interpolation spaces of $n$-tuples introduced by Favini and Lions, cf.~\cite[Section 1.(ii)]{Cwikel_Janson1987} and references therein. Proving that the spaces we will use below agree with the complex interpolation spaces of Favini and Lions would boil down to proving a reiteration theorem for the Favini-Lions interpolation method, which so far seems not to be available and is not within the scope of the current work.

In the following, we will denote the standard-$n$-symplex by $S_n$:
\begin{align*}
    S_n := \{ \Theta = (\theta_1, \dots, \theta_n) \in \mathbb R^n; ~0 \leq \theta_j \leq 1, ~j = 1, \dots n \text{ and } \sum_{j=1}^n \theta_j = 1\}.
\end{align*}
We will also write $\overset{\circ}{S}_n$ for the interior of $S_n$; i.e., the set of all tuples in $(\theta_1, \dots, \theta_n) \in S_n$ with $\theta_j \in (0, 1)$ for all $j = 1, \dots, n$.
\begin{defn}\label{def:Youngfunction_higher_interpolation}
    Let $\Phi_1, \dots, \Phi_n$ be weak $\Phi$-functions and $\Theta = (\theta_1, \dots, \theta_n) \in S_n$. Then, we define the function $[\Phi_1, \dots, \Phi_n]_{\Theta}$ by:
    \begin{align*}
        [\Phi_1, \dots, \Phi_n]_{\Theta}^{-1} = (\Phi_1^{-1})^{\theta_1} (\Phi_2^{-1})^{\theta_2} \cdots (\Phi_n^{-1})^{\theta_n}.
    \end{align*}
\end{defn}
Using this notation, we obtain:
\begin{thm}\label{prop:iterated_interpolation}
    Let $\Phi_1, \dots, \Phi_n$ be Young functions satisfying $1 < q_{\Phi_j} \leq p_{\Phi_j} < \infty$ for all $j$, $\Theta  = (\theta_1, \dots, \theta_n) \in \overset{\circ}{S}_n$ and $0 \leq k \leq n$. Then, $(n-1)$-times repeated convolution acts as a multi-linear continuous operator:
    \begin{align*}
         &s_{[\Phi_{[1]},\Phi_1]_{[\theta_1]}}^w \times \dots \times s_{[\Phi_{[1]},\Phi_k]_{[\theta_k]}}^w \times L^{[\Phi_{[1]},\Phi_{k+1}]_{[\theta_{k+1}]}}(\mathbb R^{2d}) \times \dots \times L^{[\Phi_{[1]},\Phi_n]_{[\theta_n]}}(\mathbb R^{2d}) \\
         &\quad \to \begin{cases}
            s_{[\Phi_1, \dots, \Phi_n]_\Theta}^w, &\quad k \text{ odd},\\
            L^{[\Phi_1, \dots, \Phi_n]_\Theta}(\mathbb R^{2d}), &\quad k \text{ even}.
        \end{cases}
    \end{align*}
\end{thm}
\begin{proof}
 By Proposition \ref{prop_1}, we know that $(n-1)$-times repeated convolution acts as a continuous multi-linear operator (where there are $k$ factors of Schatten-Orlicz ideals and $n-k$ factors of Orlicz-Lebesgue spaces):
    \begin{align}\addtocounter{equation}{1}
        s_{\Phi_1}^w &\times s_1^w \times \dots \times s_1^w \times L^1(\mathbb R^{2d}) \times \dots \times L^1 (\mathbb R^{2d}) \to \begin{cases}
            s_{\Phi_1}^w, &\quad k \text{ odd},\\
            L^{\Phi_1}(\mathbb R^{2d}),&\quad k\text{ even},
        \end{cases} \tag{\arabic{equation}.1}\label{eq3.1}\\
        s_{1}^w &\times s_{\Phi_2}^w \times \dots \times s_1^w \times L^1(\mathbb R^{2d}) \times \dots \times L^1(\mathbb R^{2d}) \to \begin{cases}
            s_{\Phi_2}^w, &\quad k \text{ odd},\\
            L^{\Phi_2}(\mathbb R^{2d}),&\quad k\text{ even},
        \end{cases} \tag{\arabic{equation}.2}\label{eq3.2}\\
        &\dots \notag\\
        s_{1}^w &\times s_1^w \times \dots \times s_1^w \times L^1(\mathbb R^{2d}) \times \dots \times L^{\Phi_n}(\mathbb R^{2d})\to \begin{cases}
            s_{\Phi_n}^w, &\quad k \text{ odd},\\
            L^{\Phi_n}(\mathbb R^{2d}),&\quad k\text{ even}.
        \end{cases}\tag{\arabic{equation}.n}\label{eq3.n}
    \end{align}
For simplicity of notation, let us assume in the following that $k = 0$. All the steps involved are exactly identical for general $0 \leq k \leq n$, but the notations get more involved. Furthermore, for simplyfing the notation, we will write $L^\Phi$ instead for $L^\Phi(\mathbb R^{2d})$ for the Orlicz spaces on $\mathbb R^{2d}$ for the remainder of this proof.

We first perform a multi-linear complex interpolation between \eqref{eq3.1} and \eqref{eq3.2} with respect to the parameter $\widetilde{\theta}_1 = \frac{\theta_2}{\theta_1 + \theta_2}$ to obtain the following continuous and multi-linear mapping property of $(n-1)$-times iterated convolution:
\begin{align}
    L^{[\Phi_1, \Phi_{[1]}]_{[\widetilde{\theta}_1]}} \times L^{[\Phi_{[1]}, \Phi_2]_{[\widetilde{\theta}_1]}} \times L^1 \times \dots \times L^1
    \to & L^{[\Phi_1, \Phi_2]_{[\widetilde{\theta}_1]}} .\tag{\arabic{equation}.2'}\label{eq3.2mod}
\end{align}
Next, we perform a multi-linear complex interpolation between Eqs.~\eqref{eq3.2mod} and ($\arabic{equation}.3$) with respect to the parameter $\widetilde{\theta_2} = \frac{\theta_3}{\theta_1 + \theta_2 + \theta_3}$ to obtain (where we also changed the order of interpolation in the first factor of \eqref{eq3.2mod} for convenience):
\begin{align*}
    &L^{[[\Phi_{[1]}, \Phi_1]_{[1-\widetilde{\theta}_1]}, \Phi_{[1]}]_{[\widetilde{\theta}_2]}} \times L^{[[\Phi_{[1]}, \Phi_2]_{[\widetilde{\theta}_1]}, \Phi_{[1]}]_{\widetilde{\theta}_2}} \times L^{[\Phi_{[1]}, \Phi_3]_{[\widetilde{\theta}_2]}} \times L^1 \times \dots \times L^1  \\
    &\quad \to L^{[[\Phi_1, \Phi_2]_{[\widetilde{\theta}_1]}, \Phi_3]_{[\widetilde{\theta}_2]}}.
\end{align*}
We observe that in general, for any Young function $\Phi$ and any $\rho, \nu \in (0, 1)$:
\begin{align*}
    [[\Phi_{[1]}, \Phi]_{[\rho]}, \Phi_{[1]}]_{[\nu]}^{-1}(s) &= ([\Phi_{[1]}, \Phi]_{[\rho]}^{-1})(s)^{1-\nu} s^\nu\\
    &= s^{(1-\rho)(1-\nu)} \Phi^{-1}(s)^{\rho(1-\nu)}s^\nu
\end{align*}
such that
\begin{align}\label{formula:iterated_interpolation1}
    [[\Phi_{[1]}, \Phi]_{[\rho]}, \Phi_{[1]}]_{[\nu]} = [\Phi_{[1]}, \Phi]_{[\rho(1-\nu)]}.
\end{align}
Further, (using the notation from Def.~\ref{def:Youngfunction_higher_interpolation}):
\begin{align*}
    [[\Phi_1, \Phi_2]_{[\widetilde{\theta}_1]}, \Phi_3]_{[\widetilde{\theta}_2]}^{-1} &= ([\Phi_1, \Phi_2]_{[\widetilde{\theta}_1]}^{-1})^{1-\widetilde{\theta}_2} (\Phi_3^{-1})^{\widetilde{\theta}_2}\\
    &= (\Phi_1^{-1})^{(1-\widetilde{\theta}_1)(1-\widetilde{\theta}_2)} (\Phi_2^{-1})^{\widetilde{\theta}_1(1-\widetilde{\theta}_2)}(\Phi_3^{-1})^{\widetilde{\theta}_2}\\
    &= [\Phi_1, \Phi_2, \Phi_3]_{[((1-\widetilde{\theta}_1)(1-\widetilde{\theta}_2), \widetilde{\theta}_1(1-\widetilde{\theta}_2), \widetilde{\theta}_2)]}^{-1}.
\end{align*}
Using these formulas, we obtain:
\begin{align*}
    &L^{[\Phi_{[1]}, \Phi_1]_{[(1-\widetilde{\theta}_1)(1-\widetilde{\theta}_2)]}} \times L^{[\Phi_{[1]}, \Phi_2]_{[\widetilde{\theta}_1(1- \widetilde{\theta}_2)]}} \times L^{[\Phi_{[1]}, \Phi_3]_{[\widetilde{\theta}_2]}} \times L_1 \times \dots \times L_1 \\
    &\quad \to L^{[\Phi_1, \Phi_2, \Phi_3]_{[((1-\widetilde{\theta}_1)(1-\widetilde{\theta}_2), \widetilde{\theta}_1(1-\widetilde{\theta}_2), \widetilde{\theta}_2)]}}.
\end{align*}
By induction, one now easily proves that for any $2 \leq j \leq n-1$ we have the following mapping property of $(n-1)$-fold iterated convolution:
\begin{align*}
    &L^{[\Phi_{[1]}, \Phi_1]_{[(1-\widetilde{\theta}_1)\cdot \dots \cdot (1-\widetilde{\theta}_j)]}} \times L^{[\Phi_{[1]}, \Phi_2]_{[\widetilde{\theta}_1\cdot (1-\widetilde{\theta}_2) \dot \dots \cdot (1-\widetilde{\theta}_j)]}} \times L^{[\Phi_{[1]}, \Phi_3]_{[\widetilde{\theta}_2\cdot (1-\widetilde{\theta}_3) \dots \cdot (1-\widetilde{\theta}_j)]}} \\
    &\times \dots \times L^{[\Phi_{[1]}, \Phi_{j+1}]_{[\widetilde{\theta}_{j-1}]}} \times L^1 \times \dots \times L^1 \\
    &\quad \to L^{[\Phi_1, \Phi_2, \dots, \Phi_{j+1}]_{[((1-\widetilde{\theta}_1)\cdot \dots \cdot (1-\widetilde{\theta}_j), \widetilde{\theta}_1(1-\widetilde{\theta}_2)\cdot \dots \cdot (1-\widetilde{\theta}_j), \widetilde{\theta}_2(1-\widetilde{\theta}_3) \cdot \dots \cdot (1-\widetilde{\theta}_j), \dots, \widetilde{\theta}_j)]}}.
\end{align*}
Here, we let for $2 \leq j \leq n-1$: $\widetilde{\theta}_j = \frac{\theta_{j+1}}{\theta_1 + \theta_2 + \dots + \theta_{j+1}}$. In particular, for $j = n-1$ we see that:
\begin{align*}
    &L^{[\Phi_{[1]}, \Phi_1]_{[(1-\widetilde{\theta}_1)\cdot \dots \cdot (1-\widetilde{\theta}_{n-1})]}} \times L^{[\Phi_{[1]}, \Phi_2]_{[\widetilde{\theta}_1\cdot (1-\widetilde{\theta}_2) \dot \dots \cdot (1-\widetilde{\theta}_{n-1})]}} \times L^{[\Phi_{[1]}, \Phi_3]_{[\widetilde{\theta}_2\cdot (1-\widetilde{\theta}_3) \dots \cdot (1-\widetilde{\theta}_{n-1})]}} \\
    &\times \dots \times L^{[\Phi_{[1]}, \Phi_{n}]_{[\widetilde{\theta}_{j-1}]}}\\
    &\quad \to L^{[\Phi_1, \Phi_2, \dots, \Phi_n]_{[((1-\widetilde{\theta}_1)\cdot \dots \cdot (1-\widetilde{\theta}_{n-1}), \widetilde{\theta}_1(1-\widetilde{\theta}_2)\cdot \dots \cdot (1-\widetilde{\theta}_{n-1}), \widetilde{\theta}_2(1-\widetilde{\theta}_3) \cdot \dots \cdot (1-\widetilde{\theta}_{n-1}), \dots, \widetilde{\theta}_{n-1})]}}.
\end{align*}
We will now simplify the expressions involving the $\widetilde{\theta}_j$ by backwards induction. By definition, $\widetilde{\theta}_{n-1} = \theta_n$ such that $(1-\widetilde{\theta}_{n-1}) = (1-\theta_n)$. Therefore, 
\begin{align*}
    \widetilde{\theta}_{n-2}(1-\widetilde{\theta}_{n-1}) &=  \frac{\theta_{n-1}}{1-\theta_n}(1-\theta_n) = \theta_{n-1},\\
    (1-\widetilde{\theta}_{n-2})(1-\widetilde{\theta}_{n-1}) &= (1 - \frac{\theta_{n-1}}{1-\theta_n})(1-\theta_n) = \frac{1 - \theta_{n-1} - \theta_{n}}{1-\theta_n} \cdot (1-\theta_n) \\
    &= 1 - \theta_{n-1} - \theta_n.
\end{align*}
By backwards induction, one now easily proves for $m = 1, \dots, n-1$:
\begin{align*}
    \widetilde{\theta}_m (1-\widetilde{\theta}_{m+1}) \cdot \dots \cdot (1-\widetilde{\theta}_{n-1}) = \theta_{m+1},\\
    (1-\widetilde{\theta}_m) (1-\widetilde{\theta}_{m+1}) \cdot \dots \cdot (1-\widetilde{\theta}_{n-1}) = 1 - \theta_{m+1} - \dots - \theta_n.
\end{align*}
In particular, we see that the final mapping property for $(n-1)$-times iterated convolution can be written as
\begin{align*}
    &L_{[\Phi_{[1]}, \Phi_1]^{[\theta_1]}} \times L^{[\Phi_{[1]}, \Phi_2]_{[\theta_2]}} \times \dots \times L^{[\Phi_{[1]}, \Phi_{n}]_{[\theta_m]}}\to L^{[\Phi_1, \dots, \Phi_n]_{[\Theta]}},
\end{align*}
which finishes the proof.
\end{proof}
We add the following technical result, which will allow us to derive a particularly useful instance of the above result.
\begin{prop}\label{Prop-construction-Phi-functions}
Let $\psi_0, \psi_1, \dots, \psi_n$ be Young functions such that  with the notation analogous to \eqref{GL_equivalent_Characterization_exponent} and \eqref{GL_equivalent_Characterization_exponent2} $1 < q_{\psi_j} \leq p_{\psi_j} < \infty$ for every $j = 1, \dots, n$ and for all $s>0$ they satisfy the relation
\begin{equation}\label{GL-IYF}
    s^{n-1}\psi_0^{-1}(s) = \psi_1^{-1}(s) \cdot \dots \cdot \psi_n^{-1}(s). 
\end{equation}
Further, we assume that 
\begin{equation}\label{condition}
n-1 < \sum_{j=1}^n \frac{1}{p_{\psi_j}}. 
\end{equation}
Then there is $\Theta \in \overset{\circ}{S}_n$ and there are weak $\Phi$-functions $\Phi_1, \dots \Phi_n \in \Phi_w$ satisfying $1 < q_{\Phi_j} \leq p_{\Phi_j}< \infty$ and solving the equations: 
\begin{align}\addtocounter{equation}{1}
\psi_0
&= [\Phi_1, \dots, \Phi_n]_{[\Theta]}.\tag{\arabic{equation}.0}\label{IYF-1}\\
\psi_1
&= [\Phi_{[1]}, \Phi_1]_{[\theta_1]},\tag{\arabic{equation}.1}\label{IYF-2}\\
\dots \notag\\
\psi_n
&= [\Phi_{[1]}, \Phi_n]_{[\theta_n]},\tag{\arabic{equation}.n}\label{IYF-3}
\end{align}
\end{prop}
\begin{proof}
For fixed value $\Theta \in \overset{\circ}{S}_n$ we solve (\ref{IYF-2}) to (\ref{IYF-3}) to obtain functions $\Phi_1^{-1}, \dots, \Phi_n^{-1}$ that then also solve (\ref{IYF-1}): 
\begin{equation*}
\Phi_j^{-1}(\tau)= \frac{\big{(} \psi_j^{-1}(\tau)\big{)}^{\frac{1}{\theta_j}}}{\tau^{\frac{1-\theta_j}{\theta_j}}}. 
\end{equation*}
Since $\psi_j\in \Phi_c$ satisfies $p_{\psi_j} < \infty$, and hence $\Delta_2$, it cannot attain the value $\infty$ and, by convexity, it is a continuous surjection. 
Hence for any $\tau \in (0, \infty)$ there is $s\in (0, \infty)$ with $\tau=\psi_j(s)$. 
Note that $\Phi_j$ is increasing if and only if 
\begin{equation*}
\frac{\psi_j^{-1}(\tau)}{\tau^{1-\theta_j}}= \left( \frac{\psi_j(s)}{s^{\frac{1}{1-\theta_j}}} \right)^{-(1-\theta_j)} 
\end{equation*}
is increasing which is equivalent to $\psi_j(s)/s^{\frac{1}{1-\theta_j}}$ being decreasing. With the generalized exponent $1 \leq p_{\psi_j} < \infty$ in \eqref{GL_equivalent_Characterization_exponent2}  this is the case if: 
\begin{equation*}
\frac{1}{1-\theta_j}> p_{\psi_j} \hspace{3ex} \mbox{\it equivalently} \hspace{3ex} 1-\frac{1}{p_{\psi_j}} < \theta_j. 
\end{equation*}
Since the intersection of sets
\begin{align*}
    T := \left \{ \Theta \in (0, 1)^n; ~1 - \frac{1}{p_{\psi_1}} < \theta_1 \right \} \cap \dots \cap \left \{ \Theta \in (0, 1)^n; 1 - \frac{1}{p_{\psi_n}} < \theta_n \right \}
\end{align*}
is open and non-empty and we further assume condition \eqref{condition}, which can be rephrased as
\begin{align*}
    \sum_{j=1}^n (1 - \frac{1}{p_{\psi_j}}) < 1,
\end{align*}
there exists a solution $\Theta \in T \cap \overset{\circ}{S}_n$, which we now fix in the following. With this choice of $\Theta$, all the functions $\Phi_1, \dots, \Phi_n$ are monotonously increasing.
\vspace{1mm}\par 
In order to verify that the $\Phi_j$ define weak $\Phi$-functions we have to show that they satisfy $\textup{(aInc)}_1$. Note that 
\begin{align*}
\frac{\Phi_j^{-1}(s)}{s}
&=\left( \frac{\psi_j^{-1}(s)}{s} \right)^{\frac{1}{\theta_j}}= \left( \frac{\psi_j(\tau_j)}{\tau_j} \right)^{-\frac{1}{\theta_j}}, 
\end{align*}
where $\psi_j(\tau_j)=s$. Since the $\psi_j$ are monotonously increasing and satisfy $\textup{(aInc)}_1$ with $a=1$ (according to 
(\ref{GL-convexity-and-one-decreasing})) we conclude that the quotients on the left hand side are monotonously decreasing, i.e.  $\Phi_j^{-1} \in \textup{(aDec)}_1$ with $a=1$ 
for $j = 1,\dots, n$. According to \cite[Prop. 2.3.7]{Harjulehto_Hasto2019} this is equivalent to $\Phi_j$ being $\textup{(aInc)}_1$. Hence it follows that the $\Phi_j$, $j = 1, \dots, n$, are weak $\Phi$-functions. 
\vspace{1mm}\par 
In order to show that $\Phi_j$ satisfy $p_{\Phi_j} < \infty$ is is sufficient to prove the existence of $\alpha_j < \infty$ such that 
\begin{equation*}
s \mapsto \frac{\Phi_j(s)}{s^{\alpha_j}}, \hspace{5ex} (j=1, \dots, n)
\end{equation*}
are decreasing functions (see \cite[Lemma 2.2.6]{Harjulehto_Hasto2019}). With $\theta_j$ as above and $ \tau_j= \Phi_j(s)=\psi_j(\tilde{s}_j)$ we have: 
\begin{align*}
\frac{\Phi_j(s)}{s^{\alpha_j}}
&= 
\left( \frac{\Phi_j^{-1}(\tau_j)}{\tau_j^{\frac{1}{\alpha_j}}} \right)^{-\alpha_j}=
\left( \frac{\psi_j^{-1}(\tau_j)}{\tau_j^{1-\theta_j +\frac{\theta_j}{\alpha_j}}}\right)^{-\frac{\alpha_j}{\theta_j}}
= \left( \frac{\psi_j(\tilde{s}_j)}{\tilde{s}_j^{\frac{1}{1-\theta_j + \frac{\theta_j}{\alpha_j}}}}
 \right)^{\frac{\alpha_j}{\theta_j}(1-\theta_j+ \frac{\theta_j}{\alpha_j})}. 
\end{align*}
The right hand sides are decreasing in case of: 
\begin{equation*}
\frac{1}{1-\theta_j + \frac{\theta_j}{\alpha_j}}>p_{\psi_j}, \quad j = 1, \dots, n. 
\end{equation*}
By choosing $\alpha_j>0$ sufficiently large these conditions are fulfilled since $1/(1-\theta_j)> p_{\psi_j}$. This proves that $p_{\Phi_j} < \infty$. Similarly, one deduces that $q_{\Phi_j} > 1$: With the same notation as above, we have that $q_{\Phi_j} > 1$ if there exists some $\beta_j > 1$ such that $\Phi_j(s)/s^{\beta_j}$ is increasing. Just as above, this is now equivalent to saying that
\begin{align*}
    \left( \frac{\psi_j(\tilde{s}_j)}{\tilde{s}_j^{\frac{1}{1-\theta_j + \frac{\theta_j}{\beta_j}}}}
 \right)^{\frac{\beta_j}{\theta_j}(1-\theta_j+ \frac{\theta_j}{\beta_j})}
\end{align*}
is increasing, which is true whenever
\begin{align*}
    \frac{1}{1-\theta_j + \frac{\theta_j}{\beta_j}} < q_{\psi_j}.
\end{align*}
Since we assumed $q_{\psi_j} > 1$, upon letting $\beta_j$ being close enough to $1$, the above inequality is satisfied. This finishes the proof.
\end{proof}
\begin{rem}\label{Remark_equivalent_functions}
In \cite{Harjulehto_Hasto2019} an equivalence relation of functions is defined. More precisely,  $\varphi $ and $\psi$ on $[0, \infty)$ are said to be 
equivalent, $\varphi \simeq \psi$, if there is $L\geq 1$ such that 
for $j=1,2$ and $t>0$: 
\begin{equation*}
\varphi(L^{-1} t) \leq \psi(t) \leq \varphi(Lt).
\end{equation*}
Let $\Phi_j\in \Phi_w$, $j=1,\dots, n$ be the weak $\Phi$-functions constructed in Proposition \ref{Prop-construction-Phi-functions}. 
According to \cite[Lemma 2.2.1]{Harjulehto_Hasto2019} there exist convex and left-continous $\Phi$-functions $\widetilde{\Phi}_j\in \Phi_c$ 
that are equivalent to $\Phi_j$, i.e.  $\Phi_j \simeq \widetilde{\Phi}_j$ with $q_{\widetilde{\Phi}_j} > 1$. \cite[Prop. 3.2.4]{Harjulehto_Hasto2019} implies that 
$L_{\Phi_j}(\mathbb{R}^{2d})=L_{\widetilde{\Phi}_j}(\mathbb{R}^{2d})$ and the norms of both Orlicz spaces are comparable. Moreover, 
\cite[Lemma 2.1.8 (c)]{Harjulehto_Hasto2019} in combination with \cite[Lemma 2.2.6]{Harjulehto_Hasto2019} implies that 
the doubling property $\Delta_2$ is preserved under the equivalence relation "$\simeq$". Hence, $\widetilde{\Phi}_j$ fulfills $\Delta_2$, as well and therefore 
$\widetilde{\Phi}_j$ are continuous Young functions in the sense of the present paper which also satisfy $p_{\widetilde{\Phi}_j} < \infty$.
\end{rem}
Combining Theorem \ref{prop:iterated_interpolation}, Proposition \ref{Prop-construction-Phi-functions} and Remark \ref{Remark_equivalent_functions} shows.  
\begin{thm}\label{Theorem_interpolation_by_extrapolation}
Let $\psi_0, \psi_1, \dots, \psi_n$ be Young functions with $1 < q_{\psi_j} \leq p_{\psi_j} < \infty$ for $j = 1, \dots, n$ and for all $s>0$ satisfying the relation
\begin{equation}
    s^{n-1}\psi_0^{-1}(s) = \psi_1^{-1}(s)\cdot \dots \cdot  \psi_n^{-1}(s) 
\end{equation}
and assume that (\ref{condition}) holds. Then, for any $k = 0, \dots, n$ the $(n-1)$-times iterated convolutions acts as a continuous multi-linear operator: 
\begin{align*}
        s_{\psi_1}^w \times \dots \times s_{\psi_k}^w \times L^{\psi_{k+1}}(\mathbb R^{2d}) \times \dots \times L^{\psi_n}(\mathbb R^{2d}) \to \begin{cases}
            L^{\psi_0}(\mathbb R^{2d}), \quad &k \text{ even},\\
            s_{\psi_0}^w, \quad &k \text{ odd}.
        \end{cases}
    \end{align*}
\end{thm}

\section{Dilated convolutions}\label{sec:4}

In the following, we denote for a function $a: \mathbb R^{2d} \to \mathbb C$ and a constant $0 \neq t \in \mathbb R$:
\begin{align*}
    a_t(z) := a(tz)
\end{align*}
Note that:
\begin{align*}
    a_s \ast b_t = \left( a \ast b_{t/s} \right)_s.
\end{align*}
We have the following result:
\begin{thm}[{\cite[Theorem 3.3']{Toft2002}}]\label{thm:dilated_Toft}
    Let $0 \neq t_1, \dots, t_n \in \mathbb R$ and $c_1, \dots, c_n \in \{ -1, 1\}$ such that
    \begin{align}\label{cond:tj}
         \frac{c_1}{t_1^2} + \dots + \frac{c_n}{t_n^2} = 1.
    \end{align}
    Further, let $p_1, \dots, p_n, r \geq 1$ such that
    \begin{align*}
        \frac{1}{p_1} + \dots + \frac{1}{p_n} = n-1 + \frac{1}{r}.
    \end{align*}
    Then, the map
    \begin{align*}
        \mathcal S(\mathbb R^{2d}) \times \dots \times \mathcal S(\mathbb R^{2d}) \ni (a^1, \dots, a^n) \mapsto a_{t_1}^1 \ast \dots \ast a_{t_n}^n \in \mathcal S(\mathbb R^{2d})
    \end{align*}
    continuously extends to a multi-linear map
    \begin{align*}
        s_{p_1}^w \times \dots \times s_{p_n}^w \to s_r^w
    \end{align*}
    satisfying
    \begin{align*}
        \| a_{t_1}^1 \ast \dots \ast a_{t_n}^n\|_{s_r^w} \leq (2\pi)^{d(n-1)/2}|t_1|^{-2d/p_1} \cdot \dots \cdot |t_n|^{2d/p_n} \| a^1\|_{s_{p_1}^w} \cdot \dots \cdot \| a^n \|_{s_{p_n}^w}.
    \end{align*}
\end{thm}
If we fix the $t_j$ as in the previous result, then the particular choice of the $p_j$ with $p_j = 1$ for all $j = 1, \dots, n$ gives the continuous multi-linear map:
\begin{align*}
    s_1^w \times \dots \times s_1^w \to s_1^w.
\end{align*}
Similarly, we can have all but one $p_j$ equal to one and one particular $p_k = \infty$, giving the continuous mapping
\begin{align*}
    s_1^w \times \dots \times s_1^w \times s_\infty^w \times s_1^w \times \dots \times s_1^w \to s_\infty^w,
\end{align*}
and this works for any choice of $k = 1, \dots, n$. In particular, we can now reason as in the proof of Proposition \ref{prop_1} to obtain:
\begin{prop}
    Let $\Phi$ be a Young function satisfying $1 < q_\Phi \leq p_\Phi < \infty$. Further, let $0 \neq t_j \in \mathbb R$, $j = 1, \dots, n$, satisfy condition \eqref{cond:tj}. Then, the dilated convolution
    \begin{align*}
        \mathcal S(\mathbb R^{2n}) \times \dots \times \mathcal S(\mathbb R^{2n})\ni (a^1, \dots, a^n) \mapsto a_{t_1}^1 \ast \dots \ast a_{t_n}^n \in \mathcal S(\mathbb R^{2n})
    \end{align*}
    extends for every $k =1, \dots, n$ to a continuous multi-linear map
    \begin{align*}
        s_1^w\times \dots \times s_1^w \times s_{\Phi}^w \times s_1^w \times \dots \times s_1^w \to s_\Phi^w
    \end{align*}
    with $s_\Phi^w$ at the $k$-th position of the product on the left-hand side.
\end{prop}
Having this result at hand, we are now in the position to apply the reasoning through iterated complex interpolation that was presented in the proof of Proposition \ref{prop:iterated_interpolation}. Following the same steps of that proof, one now deduces the following:
\begin{thm}
    Let $\Phi_1, \dots, \Phi_n$ be Young functions satisfying $1 < q_{\Phi_j} \leq p_{\Phi_j} < \infty$ for $j = 1, \dots, n$, $\Theta \in \overset{\circ}{S}_n$ and $0 \neq t_j \in \mathbb R$, $j = 1, \dots, n$, satisfy condition \eqref{cond:tj}. Then, the dilated convolution
     \begin{align*}
        \mathcal S(\mathbb R^{2n}) \times \dots \times \mathcal S(\mathbb R^{2n})\ni (a^1, \dots, a^n) \mapsto a_{t_1}^1 \ast \dots \ast a_{t_n}^n \in \mathcal S(\mathbb R^{2n})
    \end{align*}
    extends to a continuous multi-linear map:
    \begin{align*}
         s_{[\Phi_{[1]},\Phi_1]_{[\theta_1]}}^w \times \dots \times s_{[\Phi_{[1]},\Phi_n]_{[\theta_n]}}^w \to 
            s_{[\Phi_1, \dots, \Phi_n]_\Theta}^w.
    \end{align*}
\end{thm}
Combining this result with Proposition \ref{Prop-construction-Phi-functions}, as well as Remark \ref{Remark_equivalent_functions}, yields our final result on dilated convolutions:
\begin{thm}\label{thm:dilated_conv}
    Let $\psi_0, \psi_1, \dots, \psi_n$ be Young functions with $1 < q_{\psi_j} \leq p_{\psi_j} < \infty$ for $j = 1, \dots, n$ and for all $s>0$ satisfying the relation
\begin{equation}
    s^{n-1}\psi_0^{-1}(s) = \psi_1^{-1}(s)\cdot \dots \cdot  \psi_n^{-1}(s) 
\end{equation}
and assume that (\ref{condition}) holds. Further, assume that $0 \neq t_j \in \mathbb R$, $j = 1, \dots, n$, satisfy condition \eqref{cond:tj}. Then, the dilated convolution
\begin{align*}
        \mathcal S(\mathbb R^{2n}) \times \dots \times \mathcal S(\mathbb R^{2n})\ni (a^1, \dots, a^n) \mapsto a_{t_1}^1 \ast \dots \ast a_{t_n}^n \in \mathcal S(\mathbb R^{2n})
    \end{align*}
    extends to a continuous multi-linear map:
    \begin{align*}
         s_{\psi_1}^w \times \dots \times s_{\psi_n}^w \to 
            s_{\psi_0}^w
    \end{align*}
\end{thm}

\section{Passage to Quantum Harmonic Analysis}\label{sec:5}

In this short final section, we want to rephrase the present results in the notions of \emph{quantum harmonic analysis} (QHA) \cite{werner84}. We restrict to the simplest case, i.e., to the projective unitary representation of $\mathbb R^{2d}$ on $L^2(\mathbb R^d)$ given by the Weyl operators $W_z \varphi(t) = e^{it\xi - ix\xi/2}\varphi(t-x)$, where $z = (x, \xi) \in \mathbb R^{2d}$. Nevertheless, there is no obstruction in formulating (and proving) the results at hand for QHA of more general abelian phase spaces, as e.g.~discussed in \cite{Fulsche_Galke2025}, using the tools we have presented in Section \ref{sec:3}.

Recall that for a function $f: \mathbb R^{2d} \to \mathbb C$, a bounded linear operator $A \in \mathcal L(L^2(\mathbb R^d))$ and $z \in \mathbb R^{2d}$ we write
\begin{align*}
    \alpha_z(f) = f(\cdot - z), \quad \alpha_z(A) = W_z A W_z^\ast
\end{align*}
for the shift of a function and an operator, respectively. Further, we denote
\begin{align*}
    \beta_-(f)(z) = f(-z), \quad \beta_-(A) = UAU,
\end{align*}
where $U \in \mathcal L(L^2(\mathbb R^d))$ is the parity operator $U\varphi(t) := \varphi(-t)$. Using these notions, we can define the convolutions of functions and operators formally as
\begin{align}\label{eq:conv_qha1}
    f \ast g(w) &:= \int_{\mathbb R^{2d}} f(z) \alpha_z(g)(w)~dz,\\
    f \ast A := A \ast f &:= \int_{\mathbb R^{2d}} f(z) \alpha_z(A)~dz,\\
    A \ast B(z) &:= \operatorname{tr}(A \alpha_z(\beta_-(B))).\label{eq:conv_qha3}
\end{align}
As is well-known, the convolution of two functions $f$ and $g$ is well-defined as soon as, say, $f \in L^1(\mathbb R^{2d})$ and $g$ is either in $L^p(\mathbb R^{2d})$ for some $1 \leq p \leq \infty$. Similarly, all the convolutions described are well-defined on reasonable classes of functions and operators. Recall that we denote by $S^p(L^2(\mathbb R^d))$ the $p$-Schatten ideal on $L^2(\mathbb R^d)$ for $1 \leq p \leq \infty$, with the usual convention that $S^\infty(L^2(\mathbb R^d)) = \mathcal L(L^2(\mathbb R^d))$. The basic mapping properties of these convolutions can be summarized as:

\begin{thm}{\cite{werner84, Luef_Skrettingland2018a, Fulsche_Galke2025}}
    \begin{enumerate}
        \item Equipped with the convolution operations defined above, $L^1(\mathbb R^{2d}) \oplus S^1(L^2(\mathbb R^d))$ forms a commutative Banach algebra.
        \item Let $1 \leq p \leq \infty$. Then, the convolutions of functions and operators endow $L^p(\mathbb R^{2d}) \oplus S^p(L^2(\mathbb R^d))$ with the structure of a Banach module over $L^1(\mathbb R^{2d}) \oplus S^1(L^2(\mathbb R^d))$.
    \end{enumerate}
\end{thm}

Besides these continuity results, the general form of Young's convolution estimate is also available for the operator-function convolutions, cf.~also \cite{werner84, Luef_Skrettingland2018a, Fulsche_Galke2025}. As they turn up as a special case of the results that we are going to formulate, we defer from spelling them out explicitly here.

We want to note that for an operator $A = \mathrm{op}^w(f)$, the class $S^\Phi(L^2(\mathbb R^d))$ agrees (by definition) with the Weyl quantization of $s_\Phi^w$. Links between the two spaces, in the language of quantum harmonic analysis, can be made precise, say, on the level of tempered distributions, cf.~the above-mentioned results as well as the paper \cite{keyl_kiukas_werner16}. We will not explain this further here and refer the interested reader to the literature. Instead, we formulate the results derived in Section \ref{sec:3} in terms of operator convolutions. Proposition \ref{prop_1} now turns into the following:
\begin{prop}
    Let $\Phi$ be a Young function, satisfying $1 < q_\Phi \leq p_\Phi < \infty$. Then, the operator convolutions \eqref{eq:conv_qha1} - \eqref{eq:conv_qha3} turn $L^\Phi(\mathbb R^{2d}) \oplus S^\Phi(L^2(\mathbb R^d))$ into a Banach module over $L^1(\mathbb R^{2d}) \oplus S^1(L^2(\mathbb R^d))$, i.e., we obtain a continuous bilinear map
    \begin{align*}
        \ast: [L^1(\mathbb R^{2d}) \oplus S^1(L^2(\mathbb R^d))] \times [L^\Phi(\mathbb R^{2d}) \oplus S^\Phi(L^2(\mathbb R^d))] \to [L^\Phi(\mathbb R^{2d}) \oplus S^\Phi(L^2(\mathbb R^d))].
    \end{align*}
\end{prop}

For the remaining result, we cannot use the convenient language of modules over Banach algebras. Theorem \ref{Theorem_interpolation_by_extrapolation} translates into the following:
\begin{thm}
    Let $\psi_0, \psi_1, \dots, \psi_n$ be Young functions with $1 < q_{\psi_j} \leq p_{\psi_j} < \infty$ for $j = 1, \dots, n$. Further, assume they satisfy 
    \begin{align*}
         s^{n-1}\psi_0^{-1}(s) = \psi_1^{-1}(s) \cdot \dots \cdot \psi_n^{-1}(s)
    \end{align*}
    for all $s > 0$. Assume that \eqref{condition} holds true and fix $k \in \{ 0, \dots, n\}$. Then, the $(n-1)$-times iterated convolution of  functions and operators extends to a continuous and multi-linear map:
    \begin{align}
        &S^{\psi_1}(L^2(\mathbb R^d))  \times \dots \times S^{\psi_k}(L^2(\mathbb R^d)) \times \times L^{\psi_{k+1}}(\mathbb R^{2d})\times \dots \times L^{\psi_{n}}(\mathbb R^{2d}) \notag\\
        &\quad \to \begin{cases}
            L^{\psi_0}(\mathbb R^{2d}), \quad &k \text{ even},\\
            S^{\psi_0}(L^2(\mathbb R^d)), \quad &k \text{ odd}.
        \end{cases}
    \end{align}
\end{thm}
We note that, as dilations of symbols has no convenient phase space reformulation in terms of quantum harmonic analysis, there is no good QHA formulation for Theorem \ref{thm:dilated_conv}.

\section{Discussion}\label{sec:discussion}
Having also discussed weak Orlicz spaces in Section \ref{sec:interpolation}, it is natural to ask whether similar results can be obtained for the mapping properties of the convolution operator on weak Orlicz spaces. We will try to give some perspective on this question.

Upon studying weak $L^p$ spaces, probably one of the first observations that one encounters when studying convolution operators is that convolution is not well-defined on $L^{1}_w \times L^{1}_w$: Indeed, it is not hard to verify that the convolution of the function $\frac{1}{|x|}$ (as a function on $\mathbb R$) with itself is not well-defined. Nevertheless, having this particular obstacle in mind, almost the same results as for usual $L^p$ spaces can be obtained for weak $L^p$ spaces. One way of approaching this is using the real interpolation method instead of the complex interpolation method. Given a compatible couple of Banach spaces $(X_0, X_1)$, we will denote the space obtained by applying the real interpolation method with parameter $\theta \in (0, 1)$ and $1 \leq r \leq \infty$ as $[X_0, X_1]_{\theta, r}$ (see, e.g., \cite{Bergh_Lofstrom1976} for details on the real interpolation method). In particular for any measure space $(\Omega, \mathcal A, \mu)$ it is well-known that the real interpolation method applied to spaces of $p$-integrable functions $L^p(\mu) := L^p(\Omega, \mathcal A, \mu)$ yields Lorentz spaces $L^{p, q}(\mu)$ (cf.~\cite[Theorem 5.3.1]{Bergh_Lofstrom1976}). Specifying this to the situation $r = \infty$, and recalling that the Lorentz space $L^{p, \infty}(\mu)$ agrees with the weak $L^p$ space $L_w^p(\mu) := L_w^p(\Omega, \mathcal A, \mu)$, one has:
\begin{align}\label{real_interpolation_lp}
    [L^p(\mu), L^q(\mu)]_{\theta, \infty} = [L^{p}_w(\mu), L^q(\mu)]_{\theta, \infty} = [L^{p}_w(\mu), L^{q}_w(\mu)]_{\theta, \infty} = L^{r}_w(\mu),
\end{align}
whenever $1 \leq p, q \leq \infty$ with $p \neq q$, $\theta \in (0, 1)$ and $\frac{1}{r} = \frac{1-\theta}{p} + \frac{\theta}{q}$. Using this result, one can show that convolution acts as a bilinear continuous operator
\begin{align*}
    \ast: L^{p}_w \times L^{q}_w \to L^{r},
\end{align*}
when $1 < p, q < \infty$ and $\frac{1}{p} + \frac{1}{q} = 1 + \frac{1}{r}$. Note that the argument using the real interpolation method does not extend to $p = q = 1$, simply because $[L^1, L^1]_{\theta, \infty} = L^1$.

Having made significant use of the complex interpolation method in Sections \ref{sec:3} and \ref{sec:4}, one could of course try to adapt the same reasoning with the real interpolation method in place. And there is no problem with actually applying the method. Nevertheless, on more practical terms it turns out that it is seemingly not well-known what the real interpolation spaces between two Orlicz spaces is. One could suspect that the following holds true:
\begin{align*}
    [L^{\phi}, L^\psi]_{\theta, \infty} = [L^\phi_w, L^\psi]_{\theta, \infty} = [L^\phi_w, L^\psi_w]_{\theta, \infty} = L^{[\phi, \psi]_\theta}_w,
\end{align*}
at least under some reasonable assumption on $\phi, \psi$, at least ensuring that $L^\phi \neq L^\psi$. If such a result would hold true, then results analogous to those obtained in Sections \ref{sec:3} and \ref{sec:4} could be obtained.

To give some demonstration of this method, we show that the results on convolutions and dilated convolutions can be extended to weak Schatten ideals, using the method sketched above. In the following, we denote by $s^w_{p, \infty}$ the class of Weyl symbols of operators in the weak $p$-Schatten ideal, normed in the obvious way.
\begin{thm}
Let $1 < p_1, \dots, p_n < \infty$ and $r > 1$ such that
\begin{equation}
    \frac{1}{p_1} + \dots + \frac{1}{p_n} = \frac{1}{r} + n-1.
\end{equation}
Then, for any $k = 0, \dots, n$ the $(n-1)$-times iterated convolutions acts as a continuous multi-linear operator: 
\begin{align*}
        s_{p_1, \infty}^w \times \dots \times s_{p_k, \infty}^w \times L^{p_{k+1}}_w(\mathbb R^{2d}) \times \dots \times L^{p_n}_w(\mathbb R^{2d}) \to \begin{cases}
            L^{r}(\mathbb R^{2d}), \quad &k \text{ even},\\
            s_{r}^w, \quad &k \text{ odd}.
        \end{cases}
    \end{align*}
\end{thm}
\begin{thm}\label{thm:real_interpolated}
    Let $0 \neq t_1, \dots, t_n \in \mathbb R$ and $c_1, \dots, c_n \in \{ -1, 1\}$ such that
    \begin{align*}
         \frac{c_1}{t_1^2} + \dots + \frac{c_n}{t_n^2} = 1.
    \end{align*}
    Further, let $p_1, \dots, p_n, r> 1$ such that
    \begin{align*}
        \frac{1}{p_1} + \dots + \frac{1}{p_n} = n-1 + \frac{1}{r}.
    \end{align*}
    Then, the map
    \begin{align*}
        \mathcal S(\mathbb R^{2d}) \times \dots \times \mathcal S(\mathbb R^{2d}) \ni (a^1, \dots, a^n) \mapsto a_{t_1}^1 \ast \dots \ast a_{t_n}^n \in \mathcal S(\mathbb R^{2d})
    \end{align*}
    continuously extends to a multi-linear map
    \begin{align*}
        s_{p_1,\infty}^w \times \dots \times s_{p_n,\infty}^w \to s_{r}^w.
    \end{align*}
\end{thm}
We only prove the second of those two theorems, and also just for the bilinear case. This shows the essence of the proofs with the real interpolation method, which can be applied in the same manner for proving the first theorem. The extension to the multi-linear case is also rather straightforward with the use of the reiteration theorem, having for example the reasoning from the proof of Theorem \ref{prop:iterated_interpolation} in mind.
\begin{proof}[Proof of Theorem \ref{thm:real_interpolated}]
    For the bilinear case, note that we have:
    \begin{align*}
        s_r^w \times s_1^w &\to s_r^w\\
        s_1^w \times s_r^w &\to s_r^w.
    \end{align*}
    Using \ref{real_interpolation_lp} for sequence spaces and applying \cite[Theorem A.2.3]{Simon1976} shows that the spaces $s_p^w$ follow the same interpolation behavior as the $L^p$ spaces: $[s_{p_0}^w, s_{p_1}^w]_{\theta, \infty} = s_{p_\theta, \infty}^w$ with $\frac{1}{p_\theta} = \frac{1-\theta}{p_0} + \frac{\theta}{p_1}$.
    
    With this observations, we now have $[s_r^w, s_1^w]_{\theta, \infty} = s_{r/(\theta(r-1)+1), \infty}^w$ and $[s_1^w, s_r^w]_{\theta, \infty} = s_{r/(\theta(1-r) + r), \infty}^w$. Applying this with the bilinear version of the real interpolation method (see, e.g., \cite[Exercise 3.13.5]{Bergh_Lofstrom1976}) shows that dilated convolution maps:
    \begin{align*}
        s_{r/(\theta(r-1)+1), \infty}^w \times s_{r/(\theta(1-r) + r), \infty}^w \to s_r^w.
    \end{align*}
    With $p = r/(\theta(r-1)+1)$ and $q = r/(\theta(1-r) + r)$ the mapping property now follows. 
\end{proof}

\bibliographystyle{acm}
\bibliography{main}

@article{Liu_Wang2013,
    author = {Liu, PeiDe and Wang, MoaFa},
    title = {Weak {O}rlicz spaces: Some basic properties and their
applications to harmonic analysis},
    journal = {Sci. China, Math.},
    year = 2013,
    volume = 56,
    issue = 4,
    pages = {789-802},
}

@article{Bonino_etal2023,
    author = {Bonino, Matteo and Corisasco, Sandro and Petersson, Albin and Toft, Joachim},
    title = {FOURIER TYPE OPERATORS ON {O}RLICZ SPACES AND
THE ROLE OF {O}RLICZ {L}EBESGUE EXPONENTS},
    journal = {Mediterr. J. Math.},
    year = 2024,
    volume = 21,
    issue = 8,
    pages  = {Article number 219},
    note = {24 pp.},
}

@article{Bekjan_Chen_Liu_Jiao2011,
    author = {Bekjan, T.~N. and Chen, Z. and Liu, P. and Jiao, Y.},
    title = {Noncommutative weak {O}rlicz spaces and martingale inequalities},
    journal = {Studia Math.},
    year = 2011,
    volume = 204,
    issue = 3,
    pages = {195-212},
}

@article{Fack_Kosaki1986,
    author = {Fack, T. and Kosaki, H.},
    title = {Generalized $s$-numbers of $\tau$-measurable operators},
    journal = {Pacific J. Math.},
    volume = 123,
    issue = 2,
    year = 1986,
    pages = {269-300},
}

@incollection{Pisier_Xu2003,
    author = {Pisier, G. and Xu, Q.},
    title = {Non-commutative {$L^p$} spaces},
    booktitle = {Handbook of the geometry of Banach spaces, Vol. 2},
    publisher = {North-Holland, Amsterdam},
    year = {2003},
    pages = {1459–1517},
}

@unpublished{Xu,
    author = {Xu, G.},
    title = {Noncommutative {$L_p$}-spaces and martingale inequalities},
    note = {unpublished},
}

@BOOK{Bergh_Lofstrom1976,
	author = {Bergh, J. and L\"{o}fstr\"{o}m, J.},
	title = {{Interpolation Spaces: An Introduction}},
	year = {1976},
	publisher = {Springer Verlag, Berlin - New York},
	series = {Die Grundlehren der mathematischen Wissenschaften in Einzeldarstellungen},
	volume = {223},
}

@book{Harjulehto_Hasto2019,
    author = {Harjulehto, P. and H\"{a}st\"{o}, P.},
    title = {{Orlicz Spaces and Generalized Orlicz Spaces}},
    publisher = {Springer, Cham},
    year = 2019, 
    series = {Lecture Notes in Mathematics},
    volume = {2236},
}

@article{Simon1976,
    author = {Simon, B.},
    title = {{Analysis with weak trace ideals and the number of bound states of Schr\"{o}dinger operators}},
    journal = {Trans. Amer. Math. Soc.},
    volume = 224,
    year = 1976,
    issue = 2,
    pages = {367–380},
}

@article{werner84,
	title = {{Quantum Harmonic Analysis on Phase Space}},
	author = {Werner, R.},
	year = {1984},
	journal = {J. Math. Phys.},
	volume = {25},
	issue = {5},
	pages = {1404–1411},
}

@article{Toft2002,
    title = {Continuity properties in non-commutative convolution algebras, with applications in pseudo-differential calculus},
    author = {Toft, J.},
    year = 2002,
    journal = {Bull. Sci. Math.},
    volume = 126,
    issue = 2,
    pages = {115–142},
}

@book{Hiai2021,
    author = {Hiai, F.},
    title = {{Lectures on selected topics in von Neumann algebras}},
    publisher = {EMS Press, Berlin},
    year = {2021},
    series = {EMS Ser. Lect. Math.},
}

@article{Fulsche_Galke2025,
	author = {Fulsche, R. and Galke, N.},
	year = 2025,
	title = {Quantum Harmonic Analysis on locally
compact abelian groups},
	journal = {J. Fourier Anal. Appl.},
    volume = 31,
    pages = {article number 13},
}

@article{Luef_Skrettingland2018a,
  title={Convolutions for localization operators},
  author={Luef, F. and Skrettingland, E.},
  journal={J. Math. Pures Appl.},
  volume={118},
  pages={288--316},
  year={2018},
  publisher={Elsevier}
}

@article{keyl_kiukas_werner16,
	title = {Schwartz operators},
	author = {Keyl, M. and Kiukas, J. and Werner, R.},
	year = 2016,
	journal = {Rev. Math. Phys.},
	volume = 28,
	issue = 3,
    pages = {1630001},
}

@unpublished{Bauer_Fulsche_Toft2025,
    author = {Bauer, W. and Fulsche, R. and Toft, J.},
    title = {Convolutions of {O}rlicz spaces and {O}rlicz {S}chatten classes, with applications to {T}oeplitz operators},
    note = {preprint available at arXiv:2505.01707},
}

@article{Cwikel_Janson1987,
 author = {Cwikel, M. and Janson, S.},
 title = {Real and complex interpolation methods for finite and infinite families of {Banach} spaces},
 fjournal = {Advances in Mathematics},
 journal = {Adv. Math.},
 issn = {0001-8708},
 volume = {66},
 pages = {234--290},
 year = {1987},
}

\vspace{1cm}
\begin{multicols}{2}
\noindent
Wolfram Bauer\\
\href{bauer@math.uni-hannover.de}{\Letter ~bauer@math.uni-hannover.de}
\\
\noindent
Institut f\"{u}r Analysis\\
Leibniz Universit\"at Hannover\\
Welfengarten 1\\
30167 Hannover\\
GERMANY\\ 

\noindent
Robert Fulsche\\
\href{fulsche@math.uni-hannover.de}{\Letter ~fulsche@math.uni-hannover.de}
\\
\noindent
Institut f\"{u}r Analysis\\
Leibniz Universit\"at Hannover\\
Welfengarten 1\\
30167 Hannover\\
GERMANY\\ 
\end{multicols}

\noindent
Joachim Toft\\
\href{joachim.toft@lnu.se}{\Letter ~joachim.toft@lnu.se}
\\
\noindent
Department of Mathematics\\
Linn{\ae}us University\\
V{\"a}xj{\"o}\\
SWEDEN

\end{document}